\documentclass[11pt]{amsart}

\usepackage{amscd}
\usepackage{amsthm}
\usepackage{amsmath}
\usepackage{amssymb}
\usepackage[all]{xy}
\usepackage{graphicx}
\usepackage{verbatim}
\usepackage{color}
\usepackage{enumerate}
\usepackage{calrsfs}

\DeclareMathAlphabet{\pazocal}{OMS}{zplm}{m}{n}

\newtheorem{theorem}{Theorem}[section]
\newtheorem{lemma}[theorem]{Lemma}
\newtheorem{proposition}[theorem]{Proposition}
\newtheorem{corollary}[theorem]{Corollary}

\theoremstyle{definition}

\newtheorem{example}[theorem]{Example}

\newtheorem{remark}[theorem]{Remark}
\newtheorem{notation}[theorem]{Notation}

\usepackage{amsmath}
\usepackage{url}

\def\Jac{\mbox{\rm Jac$\:$}}

\def\O{\mathcal{O}}
\def\ff{\frak}
\def\Spec{\mbox{\rm Spec}}

\def\Max{\mbox{\rm Max}}
\def\End{\mbox{\rm End}}

\def\Zar{\mbox{\rm Zar}}

\def\patch{{\rm patch}}
\def\gen{\mbox{\rm gen}}

\def\inv{{\rm inv}}

\def\cal{\mathcal}

\def\Zar{{\rm Zar}}

\begin{document}



\title[The topology of valuation-theoretic representations]{On the topology of valuation-theoretic representations of integrally  closed domains}


\author{Bruce Olberding}

\address{Department of Mathematical Sciences, New Mexico State University,
Las Cruces, NM 88003-8001}

\begin{abstract}  Let $F$ be a field. For each nonempty subset $X$ of the Zariski-Riemann space of valuation rings of $F$, let ${{A}}(X) = \bigcap_{V \in X}V$ and ${{J}}(X) = \bigcap_{V \in X}{\ff M}_V$, where ${\ff M}_V$ denotes the maximal ideal of $V$.  We examine  connections between topological features of $X$ and the algebraic structure of the ring ${{A}}(X)$.  
We show that if ${{J}}(X) \ne 0$ and $A(X)$ is a completely integrally closed local ring that is not a valuation ring of $F$, then there is a subspace $Y$ of the space of valuation rings of $F$ that is perfect in the patch topology such that  ${{A}}(X) = {{A}}(Y)$. If any countable subset of  points is removed from $Y$, then the resulting set remains a representation of $A(X)$. Additionally, if $F$ is a countable field,   the set $Y$ can be chosen  homeomorphic to the Cantor set.  We apply these results to study properties of the ring ${{A}}(X)$ with specific focus on topological conditions that guarantee ${{A}}(X)$ is a Pr\"ufer domain, a feature that is reflected in  the Zariski-Riemann space when viewed as a locally ringed space. We also 
 classify the rings ${{A}}(X)$ where $X$ has finitely many patch limit points, thus giving a topological generalization of the class of Krull domains, one that also includes interesting Pr\"ufer domains. To illustrate the latter, we show how an intersection of valuation rings arising naturally in the study of local quadratic transformations of a regular local ring can be described using these techniques.

 \end{abstract}



\maketitle


\section{Introduction}

By a classical theorem of Krull, every integrally closed subring $A$ of a field $F$ is an intersection of valuation rings of $F$ \cite[Theorem 10.4, p.~73]{Mat}.  If also $A$ is local but not a valuation ring of $F$, then $A$ is an  intersection of infinitely many valuation rings \cite[(11.11), p.~38]{N}. If  in addition $A$ is  completely integrally closed, then any set of valuation rings of $F$ that dominate $A$ and whose intersection is $A$ must be uncountable. (See Corollary~\ref{Baire}.)
 Rather than the cardinality of representing sets of valuation rings of $F$, our focus on this article is on qualitative features of these sets. We describe these features topologically as part of a goal of connecting topological properties of subspaces $X$ of the Zariski-Riemann space $\Zar(F)$ of $F$ with algebraic features of the intersection ring $A = \bigcap_{V \in X}V$. 
 We work mostly in  the patch topology on $\Zar(F)$,  with special emphasis on the collection $\lim(X)$ of  patch limit points of the subset $X$ of $\Zar(F)$. 
 
 For each nonempty subset $X$ of $\Zar(F)$, let ${{A}}(X) = \bigcap_{V \in X}V$ and ${{J}}(X) = \bigcap_{V \in X}{\ff M}_V$, where ${\ff M}_V$ denotes the maximal ideal of $V$.  We prove in 
  Lemma~\ref{finally} that if $X$ is infinite and ${{A}}(X)$ is a local ring, then ${{J}}(\lim X) \subseteq  {{A}}(X)$.  While it is clear that ${{J}}(\lim X)$ is an ideal of ${{A}}(\lim X)$, the content of the lemma is that this radical ideal is also an ideal  of the subring ${{A}}(X)$ of $A(\lim X)$.  We use this fact to show in Theorem~\ref{cic 2} that if ${{J}}(X) \ne 0$ and $A$ is completely integrally closed and local but not a valuation ring of $F$, then ${{A}}(X) = {{A}}(\lim X)$.  In fact, it is shown that $\lim(X)$ can be replaced with a perfect space, that is, a space $Y$ for with $Y = \lim(Y)$.  
   Along with the Baire Category Theorem, this implies
  in Corollary~\ref{Baire} 
    that any countable subset of $\lim(X)$ can be removed and the resulting set remains a representation of $A$. Thus no completely integrally closed local subring of $F$ that is not a valuation ring of $F$ can be written as a countable intersection of valuation rings that dominate the ring.     
    Additionally, if $F$ is a countable field, we conclude in Corollary~\ref{Cantor} that $\lim(X)$ contains a subset $Y$ such that ${{A}}(X) = {{A}}(Y)$ and $Y$ is homeomorphic to the Cantor set. Thus in the case of a countable field $F$, every local completely integrally closed subring of $F$ that is not a valuation ring of $F$ has a representation that is homeomorphic in the patch topology to the Cantor set. (By a {\it representation} of an integrally closed subring $A$ of $F$ we mean a subset $Y$ of $\Zar(F)$ such that $A = {{A}}(Y)$.)

      In Section 4 we apply the results of Section 3 to 
    develop sufficient topological conditions for an intersection of valuation rings to be  a Pr\"ufer domain; that is, a domain $A$ for which  each localization of $A$ at a maximal ideal is a valuation domain. 
    In general the question of whether an intersection of valuation rings is a Pr\"ufer domain is very subtle; see \cite{G2, LT,  OGeo, OZR, OCpt, Roq, Rush}  and their references for various approaches to this question.
  While Pr\"ufer domains have been thoroughly investigated in Multiplicative Ideal Theory (see for example \cite{FHP, G, LM}), our interest here lies more in the point of view that these rings form the ``coordinate rings'' of  sets $X$ in $\Zar(F)$ that are non-degenerate in the
sense that every localization of ${{A}}(X)$ lies in $X$. This is expressed geometrically as asserting that $X$ lies in an affine scheme of $\Zar(F)$, where $\Zar(F)$ is viewed as a locally ringed space. We discuss this viewpoint in more detail at the beginning of Section 4. We prove for example in Corollary~\ref{limit cor} that if $X$ is a set of rank one valuation rings of $F$ such that ${{J}}(X) \ne 0$ and ${{A}}(\lim X)$ is a Pr\"ufer domain for which the intersection of nonzero prime ideals is nonzero, then ${{A}}(X)$ is a Pr\"ufer domain. In Sections~4 and~5 we prove several other results in this same spirit, where the emphasis is on determination of algebraic properties of ${{A}}(X)$ from algebraic properties of the generally larger ring ${{A}}(\lim X)$.  This continues the theme in Section 3 of using the patch limit points of $X$ to shed light on $A(X)$.

  In Section 5 we 
 focus on conditions that limit the size of $\lim X$.  
  For example, in Corollary~\ref{classify} we classify the rings ${{A}}(X)$ where $X$ is a set of rank one valuation rings in $\Zar(F)$ with finitely many patch limit points. Since such rings include the class of Krull domains that have quotient field $F$, this setting can be viewed as a topological generalization of Krull domains. But the class also includes interesting Pr\"ufer domains that behave very differently than Krull domains. 
   To illustrate the latter assertion, we show that a subspace $X$ of $\Zar(F)$ associated to a sequence of iterated quadratic transforms of a regular local ring yields a Pr\"ufer domain such as those studied in Section~4. It follows that the boundary valuation ring of a  sequence of iterated quadratic transforms of regular local rings is a localization of the order valuation rings associated to each of the quadratic transforms.


For some related examples of recent  work on the topology of the Zariski-Riemann space of a field, see   \cite{FFL, FS, OGeo, OGraz, OZR,  OCpt} and their references. Our approach is  related to that of \cite{OCpt} where it is shown that if $X$ is a quasicompact set of rank one valuation rings in $\Zar(F)$ with ${{J}}(X)\ne 0$, then ${{A}}(X)$ is a one-dimensional Pr\"ufer domain with quotient field 
 $F$ and nonzero Jacobson radical.  Similarly, in the present article we see that the topological condition of having finitely many patch limit points suffices to determine algebraic features of the intersection ring ${{A}}(X)$. In both the present article and \cite{OCpt} Pr\"ufer domains play a key role because these correspond to affine schemes in $\Zar(F)$ when the Zariski-Riemann space is viewed as a locally ringed space; see \cite{OZR} for more on this point of view. In general, the determination of which subsets of $\Zar(F)$ underlie affine schemes in the Zariski-Riemann space requires geometrical in addition to topological information \cite{OGeo, OZR}.

\medskip

{\it Notation and terminology.}
 Let $X$ be a topological space, and let
  $Y$ be a subspace of $X$.  Then $x \in X$ is a {\it limit point of $Y$} (or {\it accumulation point of $Y$}) if each open neighborhood $U$ of $x$ in $X$ contains a point in $Y$ distinct from $x$. If also $X$ is Hausdorff (and the limit points we consider in this article will all be limit points in a Hausdorff topology), then every open neighborhood of the limit point $x$ contains infinitely many points of $Y$.  
   An element
$x \in X$ is an {\it isolated point of $X$} if $\{x\}$ is an open subset of $X$.
   Thus $x \in X$ is a limit point in $X$ if and only if $x$ is not an isolated point in $X$.  The set of all  isolated points of $X$, as a union of open sets, is open in $X$. Thus the set of all limit points of $X$ is closed in $X$, and the closure of a subset $Y$ in $X$ consists of the elements of $Y$ and the limit points of $Y$.

\section{Properties of patch limit points}

We assume throughout the paper that $F$ denotes a field and $\Zar(F)$ is the 
 set of valuation rings of $F$, i.e., those subrings $V$ of $F$ such that for each $0 \ne t \in F$, $t \in V$ or $t^{-1} \in V$. We view $\Zar(F)$ as a topological space under the Zariski and patch topologies (and, in Section~4, under the inverse topology), each of which will be defined using 
 the following notation.

\begin{notation} { 
For each subset $S$ of $F$, let \begin{center} ${\cal U}(S) = \{V \in \Zar(F):S \subseteq V\}$ \: and \: ${\cal V}(S) = \{V \in \Zar(F):S \not \subseteq V\}$. \end{center} }
\end{notation}

The {\it Zariski topology} on $\Zar(F)$ has as a basis of  open sets the empty set and sets of the form ${\cal U}(x_1,\ldots,x_n)$, where  $x_1,\ldots,x_n \in F$. The set $\Zar(F)$ with the Zariski topology is the {\it Zariski-Riemann space} (or {\it abstract Reimann surface}) of $F$ \cite[Chapter VI, Section 17]{ZS}.
Several authors (see for example \cite{DF} and the discussion in \cite[Remark 3.2]{OZR}) have established independently that the Zariski-Riemann space $\Zar(F)$ is a {\it spectral space}, meaning that (a) $\Zar(F)$ is quasicompact and $T_0$, (b) $\Zar(F)$ has a basis of quasicompact open sets, (c) the intersection of finitely many quasicompact open sets in $\Zar(F)$ is quasicompact, and (d) every nonempty irreducible closed set in $\Zar(F)$ has a unique generic point.  

The {\it patch topology} (or {\it constructible topology}) 
on $\Zar(F)$ is given by the topology that  has as a basis the subsets of $\Zar(F)$  of the form  \begin{center}
(a) ${\cal U}(x_1,\ldots,x_n)$,  (b) ${\cal U}(x_1,\ldots,x_n)  \cap {\cal V}(y_1) \cap \cdots \cap {\cal V}(y_m)$  or (c) ${\cal V}(y_1)$,
\end{center}
where $ x_1,\ldots,x_n,y_1,\ldots,y_m \in F$. The complement in $\Zar(F)$ of any set in this basis is again open in the patch topology. Thus  the patch topology has a basis of sets that are both closed and open  (i.e., the patch topology is zero-dimensional). The  patch topology is also spectral and hence  quasicompact
 \cite[Section 2]{Hoc}.    Unlike the Zariski topology on $\Zar(F)$, the patch topology  is always  Hausdorff.  For examples of  recent work on the patch topology of $\Zar(F)$, see \cite{FFL, OZR, OCpt}.

If $y_1,\ldots,y_m$ are nonzero, then  $V \in \Zar(F)$ is a member of a set of the form in (b) if and only if $x_1,\ldots,x_n\in V$ and $1/y_1,\ldots,1/y_n \in {\ff M}_V$.  On the other hand, if some $y_i$ is $0$, then the set in (b) is empty. From these observations, we deduce that for a nonempty subset $X$ of $\Zar(F)$, a valuation ring $U \in \Zar(F)$ is a patch limit point of $X$ if and only if 
for all $x_1,\ldots,x_n \in U$, $y_1,\ldots,y_m \in {\ff M}_U$, there exists (infinitely many, since the patch topology is Hausdorff) $ V \in X$ with $V \ne U$,  $x_1,\ldots,x_n \in V$ and $y_1,\ldots,y_m \in {\ff M}_V$.  
We will only consider limit points in $\Zar(F)$ with respect to the patch topology. 

\begin{notation} {For a subset $X$ of $\Zar(F)$ we denote by  $\lim(X)$ the 
 set of limit points of $X$ in the patch topology of $\Zar(F)$, and we 
  denote by $\patch(X)$ the closure of $X$ in the patch topology.  Thus $\patch(X) = X \cup \lim(X),$ and $X$ is patch closed if and only if $X = \patch(X)$, if and only if $\lim(X) \subseteq X$.}
  \end{notation}

  \begin{remark} \label{min notation} {Every nonempty patch closed subset  $X$ of $\Zar(F)$ is a spectral space and hence contains minimal elements with respect to set inclusion \cite[Section 2 and Proposition 8]{Hoc}. We denote by $\min(X)$ the set of minimal elements of $X$.   }
  \end{remark}
  
  Note that $\lim(X)$ is a patch closed set. In Section 4 we work with the set $\min(\lim X)$.

We now  develop some  of the technical properties of patch limit points that are needed in the next sections. The first such property is an algebraic interpretation of the ring formed by intersecting the patch limit valuation rings of a subspace of $\Zar(F)$.  Recall from the introduction the operators (or more precisely, presheaves) ${{A}}(-)$ and ${{J}}(-)$.




\begin{proposition} \label{not empty}  \label{union} Let $X$ be an infinite subset of $\Zar(F)$, and let ${\cal C}$ be the collection of cofinite subsets of $X$.  Then $\lim(X)$ is nonempty and
$$\bigcup_{Y \in {\cal C}} A(Y) \: = {{A}}(\lim X) \:\: {\mbox{ and }} \:\:
\bigcup_{Y \in {\cal C}} {{J}}(Y) \: = {{J}}(\lim X).$$

\end{proposition}

\begin{proof}  
 Since $\Zar(F)$ is quasicompact in the patch topology, so is the patch closed subset $\patch(X)$. Thus  $\lim(X)$ is nonempty since $X$ is infinite. 
 We first claim that $\bigcup_{Y \in {\cal C}}{{A}}(Y) \subseteq {{A}}(\lim X)$.
Let $Y \in {\cal C}$, and let   $U \in \lim(X)$.  Suppose $x \in F$ with $x \not \in U$. Then $U \in  {\cal V}(x)$. Since $U$ is a patch limit point of $X$ and $ {\cal V}(x)$ is a patch open subset of $X$,  there are infinitely many members of $X$ in ${\cal{V}}(x)$.  Thus, since $Y$ is cofinite, there exists $W \in Y \cap {\cal V}(x)$, and hence $x \not \in W$. Therefore, $x \not \in {{A}}(Y)$.  This proves the first claim.  

To prove the reverse inclusion, let $x \in F$ with  $x \not \in \bigcup_{Y \in {\cal C}} {{A}}(Y)$.  For each $Y \in {\cal C}$, there exists $V_Y \in Y$ such that $x \not \in V_Y$.  If the collection $Z=\{V_Y:Y \in {\cal C}\}$ is finite, then $Z' = X \setminus Z \in {\cal C}$, so that $V_{Z'} \in Z' \cap Z = \emptyset$, which is impossible. Thus $Z$ is infinite, which  as noted at the beginning of the proof implies that  
 there is $U \in \lim(Z)$. If $x \in U$, then $U \in {\cal U}(x)$, and, since $U$ is a patch limit point of $Z$, $Z \cap {\cal U}(x) \ne \emptyset$. This implies that 
 there exists $Y \in {\cal C}$ such that $x \in V_Y$, a contradiction. Thus $x \not \in U$ so that $x \not \in {{A}}(\lim X)$. This proves that 
$\bigcup_{Y \in {\cal C}}{{A}}(Y)  = {{A}}(\lim X).$ 

The proof for ${{J}}$ is similar. We verify the inclusion ``$\subseteq$'' first.  Let $Y \in {\cal C}$, and let   $U \in \lim(X)$.  Suppose $x \in F$ with $x \not \in {\ff M}_U$. Since $U$ is a valuation ring, this implies 
  $U $ is in the patch open set $ {\cal U}(x^{-1})$. Since $U \in \lim(X)$, the set ${\cal U}(x^{-1})$ contains infinitely many valuation rings $V$ in  $X$. For each such valuation ring $V$, $x \not \in {\ff M}_V$.  Thus, since $Y$ is cofinite, there exists $V \in Y$ such that $x \not \in {\ff M}_V$.  Hence $x \not \in {{J}}(Y)$.  This shows that 
$\bigcup_{Y \in {\cal C}} {{J}}(Y)    \subseteq {{J}}(\lim X).$

Finally, to prove the reverse inclusion,  let $x \in F$ with  $x \not \in \bigcup_{Y \in {\cal C}}{{J}}(Y) $.  Then for each $Y \in {\cal C}$, there exists $V_Y \in Y$ such that $x \not \in {\ff M}_{V_Y}$.  As above, the collection  
 $Z=\{V_Y:Y \in {\cal C}\}$  is infinite and so there exists $U \in \lim(Z)  $. If $x \in {\ff M}_U$, then, since $U$ is a patch limit point of $Z$, there exists $Y \in {\cal C}$ such that $x \in {\ff M}_{V_Y}$, a contradiction. Thus $x \not \in {\ff M}_U$, and it follows that $\bigcup_{Y \in {\cal C}} {{J}}(Y) = {{J}}(\lim X).$ 
  \end{proof}

\begin{corollary} \label{dominate}  Let $X$ be an infinite subset of   $ \Zar(F)$, and let $x \in F$. If $x \in {\ff M_V}$ for  all but finitely many valuation rings $V \in X$, then $x \in {{J}}(\lim X)$.  
\end{corollary}

\begin{proof} If $x \in {\ff M}_V$ for all but finitely many valuation rings in $ X$, then the set $Y = \{V \in X:x \in {\ff M}_V\}$ is cofinite. By Proposition~\ref{union}, $x \in {{J}}(Y) \subseteq {{A}}(\lim X).$
\end{proof}

In later sections we frequently work with nonempty subsets $X$ of $\Zar(F)$ such that ${{J}}(X) \ne 0$. The next corollary shows this criterion is reflected in the patch topology of $\Zar(F)$. 

\begin{corollary} \label{F cor} Let $X$ be a nonempty subset of $\Zar(F)$.   Then ${{J}}(X) = 0$ if and only if there exists $U \in \patch(X)$ such that ${\ff M}_U \cap {{A}}(X) = 0$.   
\end{corollary}

\begin{proof} 
Suppose that ${{J}}(X) = 0$. Then, for each $0 \ne a \in A:={{A}}(X)$, the set $$X_a:= \{V \in \patch(X):a \not \in {\ff M}_V\} = \patch(X) \cap {\cal U}(1/a)$$ is a nonempty  patch closed subset of $\patch(X)$.  Since $A = A(\patch(X))$, for all nonzero $a_1,\ldots,a_n \in A$ we have $X_{a_1} \cap \cdots \cap X_{a_n} = X_{a_1\cdots a_n}$, so the collection $\{X_a:0 \ne a \in A\}$ has the finite intersection property.  
As a patch closed subset of $\Zar(F)$, $\patch(X)$ is quasicompact, so there exists $U \in \bigcap_{0 \ne a \in A}X_a$. For this choice of $U$ we have ${\ff M}_U \cap A = 0$.  
Conversely, if $U \in \patch(X)$ such that ${\ff M}_U \cap A = 0$, then  Proposition~\ref{union} implies  
${{J}}(X)   \subseteq  {\ff M}_U \cap A = 0,$ which proves the corollary.
\end{proof}

In the case in which $A$ has quotient field $F$, we obtain  a simple topological criterion for when ${{J}}(X) = 0$.

\begin{corollary} \label{F cor 2} Let $X$ be a nonempty subset of $\Zar(F)$ such that ${{A}}(X)$  has quotient field $F$.    Then ${{J}}(X)= 0$ if and only if  $F \in \patch(X)$.   
\end{corollary}

\begin{proof}  Suppose ${{J}}(X) = 0$. By Corollary~\ref{F cor}, there is $U \in \patch(X)$ such that ${\ff M}_U \cap {{A}}(X) = 0$. Since $A(X)$ has quotient field $F$, this implies ${\ff M}_U = 0$ and hence $F = U \in \patch(X)$.  The converse follows from Corollary~\ref{F cor}.   
\end{proof}

A collection $X$ of valuation rings in $\Zar(F)$ has {\it finite character} if  each nonzero $x \in F$ is a unit in all but finitely many valuation rings of $X$.  This notion will be useful in the last section for the  classification  of  subsets of $\Zar(F)$ with finitely many limit points. 

\begin{corollary} \label{FC char} A nonempty subset $X$ of $\Zar(F)$ has finite character if and only if $X$ is finite or $\lim(X) = \{F\}$. 
\end{corollary}

\begin{proof}  
The direction ``$\Rightarrow$'' is proved in \cite[Lemma 4.3]{FFL} using ultrafilter methods. We include a proof here that uses the patch topology.  
Suppose that $X$ has finite character and is infinite. 
By Proposition~\ref{not empty}, 
 $\lim(X)$ is nonempty. Let $V \in \lim(X)$. Suppose $V \ne F$
 and choose $0 \ne t \in {\ff M}_V$.  
 Then $\{W \in X: t  \in {\ff M}_W\}$ is a patch open subset of $X$ containing $V$.  Since $X$ has finite character, this set is finite.  However,  $X$ is Hausdorff in the patch topology, and so  every open neighborhood of the patch limit point $V$ contains infinitely many members of $X$. This contradiction shows that  $V = F$.

Conversely, suppose that $X$ is infinite and $\lim(X)= \{F\}$. Suppose  that there exists $0 \ne t \in F$ 
and infinitely many valuation rings $V$ in $X$ 
such that $t$ is not a unit in 
  $V$. 
    Then there are infinitely many valuation rings $V \in X$ such that either  $t \in {\ff M}_V$ or $t^{-1} \in {\ff M}_V$.  Choosing  $x = t$ or $x = t^{-1}$ as needed, we conclude that  there exists $0 \ne x \in F$ such that 
   $x \in {\ff M}_V$  for infinitely many $V \in X$.  By Proposition~\ref{union}, the infinite set $\{V \in X:x \in {\ff M}_V\}$ has a patch limit point, which, since $\lim(X) = \{F\}$, must be $F$.  However, by Proposition~\ref{union}, $x$ is in the maximal ideal of any patch limit point of $\{V\in X:x \in {\ff M}_V\}$, so $x = 0$ since the patch limit point $F$ is a field. This contradiction show that $t$ is a unit in all but finitely many valuation rings in $X$, and we conclude that $X$ has finite character.  
\end{proof}

A strategy of proof in Lemma~\ref{finally} is to replace a collection $X=\{W_i:i \in I\}$ of valuation rings in $\Zar(F)$ with valuation rings $Y=\{V_i:i \in I\}$ such that $V_i \subseteq W_i$ for each $i$. The rest of this section is devoted to proving Proposition~\ref{going up}, which describes a relationship between $\lim(X) $ and $\lim(Y)$ in such a circumstance. To this end, we recall a helpful interpretation of the patch topology on $\Zar(F)$  due to 
Finocchiaro, Fontana and Loper \cite{FFL}.   An {\it ultrafilter}  on a set $X$ is a collection ${\cal F}$ of subsets of $X$ such that 
(a) $\emptyset \not \in {\cal F}$,  (a) ${\cal F}$ is closed under finite intersections, (b)  if $A \subseteq B \subseteq X$ and $A  \in {\cal F}$, then $B \in {\cal F}$, and (c) for each $A \subseteq X$, either $A \in {\cal F}$ or $B \in {\cal F}$.  The ultrafilter $\cal F$ is {\it principal} if there is a member of $X$ (necessarily unique) that is an element in every
set in $\cal F$.  
 
\begin{notation} {Let $X$ be a nonempty subset of $\Zar(F)$, and let ${\cal F}$ be an ultrafilter on $X$.  Define $V_{\cal F} = \{x \in F:\{W \in X:x \in W\} \in {\cal F}\}.$ Then $V_{\cal F}$ is a valuation ring with quotient field $F$ \cite[Corollary 3.8]{FFL}.}
\end{notation}

We use the valuation rings $V_{\cal F}$ as technical tools in proving Proposition~\ref{going up}. Specifically, we use the fact that the patch closure of $X$ in $\Zar(F)$ can  be interpreted via these valuation rings as $\patch(X)=\{V_{\cal F}:{\cal F}$ is an ultrafilter on $X\}$ \cite[Corollary 3.8]{FFL}.  The next lemma gives a similar interpretation for $\lim(X)$ involving nonprincipal ultrafilters. 

\begin{lemma} \label{uf limit} Let $X$ be an infinite subset of $\Zar(F)$, and let $V \in \Zar(F)$. 

\begin{enumerate}[$(1)$]
\item[{\em (1)}] If $V \in \lim(X)$, then $V = V_{\cal F}$ for some 
nonprincipal ultrafilter  ${\cal F}$  on $X\setminus \{V\}$.

\item[{\em (2)}]  If $V = V_{\cal F}$ for a nonprincipal ultrafilter ${\cal F}$ on an infinite subset $Y$ of $X$, then $V \in \lim(Y)$ and hence $V \in \lim(X)$. 
\end{enumerate}
\end{lemma}

\begin{proof} (1) Since $V \in \lim(X)$, $V $ is in the patch closure of $X\setminus \{V\}$. Thus, as discussed before the lemma,   there is an ultrafilter ${\cal F}$ on $X \setminus \{V\}$ such that $V = V_{\cal F}$.  This ultrafilter is nonprincipal by 
\cite[Example 2.3]{Fin} and \cite[Remark 3.2]{FFL}.


(2) Suppose $V = V_{\cal F}$ for a nonprincipal ultrafilter ${\cal F}$ on an infinite subset $Y$ of $X$.  We show every patch open subset of $Y$ containing $V$ contains infinitely many elements of $Y$.  Every patch open subset of $Y$ containing $V$ contains a basic open neighborhood of the form  ${\cal N} = {\cal U}(x_1,\ldots, x_n) \cap {\cal V}(y_1) \cap \cdots \cap V(y_m)$,  where $x_1,\ldots,x_n,y_1,\ldots,y_m \in F$.  We show $Y \cap {\cal N}$ is infinite. 
Since $x_1,\ldots,x_n \in V_{\cal F}$ and the filter ${\cal F}$ is closed under finite intersections, we have $$Y \cap {\cal U}(x_1,\ldots,x_n) \: = \: \bigcap_{i=1}^n \: \{U \in Y:x_i \in U\} \: \in \: {\cal F}.$$
Since  $y_i \not \in V_{\cal F}$ for each $i$, we have 
$\{U \in Y:y_i  \in U\} \not \in {\cal F}.$ 
  Again since ${\cal F}$ is an ultrafilter on $Y$,  this implies $$Y \cap {\cal V}(y_1) \cap \cdots \cap {\cal V}(y_m) = \bigcap_{i=1}^m\{U \in Y:y_i \not \in U\} \in {\cal F}.$$
  Since ${\cal F}$ is closed under finite intersections, we conclude that  $$Y \cap {\cal N} = Y \cap {\cal U}(x_1,\ldots,x_n) \cap {\cal V}(y_1) \cap \cdots \cap {\cal V}(y_m) \in {\cal F}.$$ Since ${\cal F}$ is a nonprincipal ultrafilter, ${\cal F}$ contains no finite subsets and hence $Y \cap {\cal N}$ is infinite, as claimed.
\end{proof}

\begin{proposition} \label{going up}  Let  $Y=\{V_i:i \in I\}$ and $X = \{W_i:i \in I\}$
be  subsets of $\Zar(F)$ such that $V_i \subseteq  W_i$ for all $i \in I$. (The $V_i$ are assumed to be distinct, as are the $W_i$.)

\begin{enumerate}[$(1)$]
\item 
For each  $W \in \lim(X)$ there exists $V \in \lim(Y)$ such that $V \subseteq W$.


\item 
For each $V \in \lim(Y)$ there exists $W \in \lim(X)$ such that $V \subseteq W$.  
\end{enumerate}
 
\end{proposition}  

\begin{proof}  (1) If $I$ is finite (so that $X$ is finite), then $\lim(X)$ is empty  and there is nothing to show. Thus we assume  that $I$ is infinite. 
Let $W \in \lim(X)$.  
If $W \in X$, so that $W = W_i$ for some $i$, let $I' = I \setminus \{i\}$. Otherwise, if
 $W \not \in X$, let $I' = I$. 
 By Lemma~\ref{uf limit}(1), there exists a nonprincipal ultrafilter ${\cal F}$ on $I'$ such that $$W = \{x \in F:\{i \in I':x \in W_i\} \in {\cal F}\}.$$  
   Let $$V = \{x \in F:\{i \in I':x \in V_i\} \in {\cal F}\}.$$ Since ${\cal F}$ is a nonprincipal ultrafilter on $I'$ and the valuation rings $V_i$ are distinct,  Lemma~\ref{uf limit}(2) implies that $V \in  \lim(Y)$.  
Finally, if $x \in V$, then $\{i \in I':x \in V_i\} \in {\cal F}$.   
 Since $V_i \subseteq W_i$ for each $i \in I'$, we have $\{i \in I':x \in V_i\} \subseteq \{i \in I':x \in W_i\}$. Since ${\cal F}$ is a filter, $\{i \in I':x \in W_i\} \in {\cal F}$. Thus $x \in W$,  
 proving that $V \subseteq W$.  
 %
 %

(2) The proof is similar to (1).   Let $V \in \lim(Y)$. If $V \in Y$, say  $V = V_i$, let $I' = I \setminus \{i\}$; otherwise, if
 $V \not \in Y$, let $I' = I$. 
  By Lemma~\ref{uf limit}(1)  there is a nonprincipal ultrafilter ${\cal F}$ on $I'$ such that 
$$V = \{x \in F:\{i \in I':x \in V_i\} \in {\cal F}\}.$$  
  Let 
$$W = \{x \in F:\{i \in I':x \in W_i\} \in {\cal F}\}.$$  
  By Lemma~\ref{uf limit}(2),  $W \in \lim(X)$. As in (1), we have $V \subseteq W$.     
\end{proof}

\section{Limit points in local intersections}

 In this section we consider local domains that are intersections of valuation rings in $\Zar(F)$.  One of the main results of the section, Lemma~\ref{finally}, shows that the limit points of a representation $X$ of $A$ determine a radical ideal of $A$.  If this radical ideal is nonzero, this in turn implies that the intersection of the patch limit points of $X$ is a fractional ideal of $A$. This technical observation is the basis for the applications in the rest of the article.

   
   Because it is needed in  the first two lemmas in this section,
we recall the notion of a projective model.  Let $D$ be a subring of the field $F$, and let $t_0,t_1,\ldots,t_n$ be nonzero elements of $ F$. For each $i$, let $D_i =D[t_0/t_i,\ldots,t_n/t_i]$, and let $${\cal M} \: = \: \bigcup_{i=0}^n \: \{(D_i)_P: P \in \Spec(D_i)\}.$$ Then ${\cal M}$ is {\it the projective model of $F/D$ defined by $t_0,t_1,\ldots,t_n$.}    Alternatively, a projective model of $F/D$ can be viewed as a projective integral scheme over $\Spec(D)$ whose function field is a subfield of $F$. 
For each local ring $B$ in  ${\cal M}$ there is a valuation ring $V$ in $\Zar(F)$ that {\it dominates} $B$ (i.e., the maximal ideal of $B$ is  a subset of ${\ff M}_V$), and for each $V \in \Zar(F)$ there exists a unique $B \in {\cal M}$ such that $V$ dominates $B$; see \cite[pp.~119--120]{ZS} or apply the valuative criterion for properness \cite[Theorem 4.7, p.~101]{Hart}. 
 For a nonempty subset $X$ of $\Zar(F)$, we denote by ${\cal M}(X)$ the set of all $B \in {\cal M}$ such that $B$ is dominated by some $V \in X$; that is, ${\cal M}(X)$ is the image of $X$ under the continuous closed surjection $\Zar(F) \rightarrow {\cal M}$ given by the domination mapping \cite[Lemma 5, p.~119]{ZS}.

\begin{lemma} {\em \cite[Lemma 3.1 and Theorem 3.2]{OCpt}} \label{principalization}   \label{trans}
Let $\emptyset \ne X \subseteq \Zar(F)$,   let $D$ be a subring of $F$ and let $t_0,\ldots,t_n \in F \setminus \{0\}$.   
\begin{enumerate}[$(1)$] 

\item Suppose ${{A}}(X)$ 
 is a local ring. Let 
  ${\cal M}$ be 
  the projective model of $F/D$ defined by $t_0,\ldots,t_n.$ If ${\cal M}(X)$ is finite, then $t_0/t_i,\ldots,t_n/t_i \in {{A}}(X)$ for some $i = 0,1,\ldots,n$.

\item Suppose
  $t _0\not \in {{A}}(X)$, $1/t_0 \not \in {{A}}(X)$,   
 $D$ is  local, and 
  all but at most finitely many $V \in X$ dominate $D$.  
 Then there exists $U \in \lim(X)$ such that $t_0,1/t_0 \in U$, $U$ dominates $D$ and  the image of $t_0$ in $U/{\ff M}_U$ is transcendental over the residue field of $D$. 
 
 \end{enumerate}

\end{lemma}

 \begin{lemma} \label{finally} 
If $X $ is an infinite subset of   $ \Zar(F)$ such that ${{A}}(X)$ is a local ring, then ${{J}}(\lim X) \subseteq {{A}}(X).$ 
%
%

\end{lemma}

 \begin{proof}   
 By Proposition~\ref{not empty}, the set $\lim(X)$  is nonempty since $X$ is infinite.   Let 
 $A = {{A}}(X).$ Most of the proof is devoted to a reduction to the case in which every valuation ring in  $X$ dominates $A$.  To make this reduction  we 
 construct a set $Y$ of valuation rings in $\Zar(F)$ that dominate $A$ and such that $Y$ has the property  that   $${{J}}(\lim X) \: \subseteq \: {{J}}(\lim Y) \: \subseteq \: A.$$

By Proposition~\ref{union}, all the valuation rings in $\lim(X)$ contain $A$. Also, since the patch topology on $\Zar(F)$ is Hausdorff, $X$ and its patch closure have the same limit points in $\Zar(F)$. Thus the ideal ${{J}}(\lim(X))$  is the same regardless of whether $X$ or its patch closure is used. 
 We may assume without loss of generality that $X$ is patch closed in  $\Zar(F)$.  

Since $X$ is patch closed,  
 $X$ has minimal elements with 
 respect to set inclusion (see Remark~\ref{min notation}). Let $\{W_i:i \in I\}$ denote the set of minimal elements of $X$.  For each $i \in I$ we have 
 $$A/(A \cap {\ff M}_{W_i}) \cong (A + {\ff M}_{W_i})/{\ff M}_{W_i} \subseteq  W_{i}/{\ff M}_{W_i}.$$  Chevalley's Extension Theorem \cite[Theorem 3.1.1, p.~57]{EP} implies there is a valuation ring $T_i$ of the field $W_i/{\ff M}_{W_i}$ such that $T_i$ dominates the local ring $(A + {\ff M}_{W_i})/{\ff M}_{W_i}$.
Let  $V_i$ be the subring of $W_i$ such that $V_i/{\ff M}_{W_i} = T_i$. Then $V_i$ is a valuation ring in $\Zar(F)$ such that 
 $V_i \subseteq W_i$ and $V_i$ dominates $A$. 
  Set $Y = \{V_i:i \in I\}$.  Since the $W_i$ are the minimal elements of $X$ and  $A \subseteq V_i \subseteq W_i$ for each $i$, we have $$A \: = \: \bigcap_{V \in X}V \:= \: \bigcap_{i \in I}W_i \: = \: \bigcap_{i \in I}V_i.$$
 
 Suppose the index set $I$ is finite. Then $A$ is a finite intersection of the valuation rings $W_i$, and hence, as a local ring, $A$ is a valuation ring with quotient field $F$  \cite[(11.11), p.~38]{N}. In this case, for every valuation overring $V$ of the valuation ring $A$, we have ${\ff M}_V \subseteq A$, so that the lemma is clear. 
 
 We assume now  that  the index set $I$ is infinite. By Proposition~\ref{not empty}, $\lim(Y)$ is nonempty. We claim  that ${{J}}(\lim X) \subseteq {{J}}(\lim Y)$. 
The set  $Y=\{W_i:i \in  I\}$ is infinite and  consists of incomparable valuation rings.  This implies that the $V_i$ are distinct, since if $V_i = V_j$ for some $i,j \in I$, then $V_i \subseteq W_i \cap W_j$. The overrings of a valuation ring form a chain under set inclusion, so necessarily $W_i$ and $W_j$ are comparable, which forces $i = j$. 
Therefore, the $V_i$ are distinct, as are the $W_i$.  Proposition~\ref{going up}(2) implies that for each $V \in \lim(Y)$ there is $W \in \lim(X)$ such that $V \subseteq W$ and  hence ${\ff M}_W \subseteq {\ff M}_V$.  
It follows that ${{J}}(\lim X) \subseteq {{J}}(\lim Y)$. 
  
Therefore,  to complete the proof that ${{J}}(\lim X) \subseteq A$, it suffices to show  ${{J}}(\lim Y) \subseteq A$.  Since  $A = \bigcap_{i \in I}V_i$,  we have that $Y = \{V_i:i \in I\}$ is an infinite representation of $A$ such that each valuation ring in $Y$ dominates $A$.    
  Let $t \in F$ such that $t \not \in A$. We show that $t \not \in {{J}}(\lim Y)$ by considering two cases. 

\smallskip  
  
{\textsc{Case 1}}:  $1/t \in A$.

\smallskip

 In this case, since $t \not \in A$, $1/t$ is not a unit in $A$. Hence $1/t$ is an element of the maximal ideal of $A$.   Since  each valuation ring in $Y$ dominates $A$, Corollary~\ref{dominate} implies every valuation ring in $\lim(Y)$ dominates $A$. Since $1/t$ is in the maximal ideal of $A$, we conclude that $1/t 
 \in {{J}}(\lim Y)$. Thus $t \not \in  {{J}}(\lim Y)$.  
 
 \medskip
 
 {\textsc{Case 2}}:  $1/t \not \in A$.
 
 \smallskip

 Let ${\cal M}$ be the projective model of $F/A$ defined by $1,t$. Since $t \not \in A$ and $1/t \not \in A$, Lemma~\ref{principalization}(1) implies that  the image ${\cal M}(Y)$ of $Y$ in ${\cal M}$ is infinite. 
Since every valuation ring in $Y$ dominates $A$, Lemma~\ref{trans}(2)  (applied to $D= A$) implies there exists $U \in \lim(Y)$ such $t$ is a unit in $U$.  
Consequently, $t \not \in {{J}}(\lim Y)$. 

\smallskip

We have shown that in both cases if $t \not \in A$ then we have $t \not \in {{J}}(\lim Y)$. It follows that ${{J}}(\lim X) \subseteq {{J}}(\lim Y) \subseteq A$, which completes the proof. 
\end{proof}


    
\begin{notation} {For $A$  a subring of $F$ and $J$  an ideal of $A$, we
 let $$ \End(J) = \{t \in F:tJ \subseteq J\}.$$ If $J \ne 0$, then $A \subseteq \End(J) \subseteq Q(A)$, where $Q(A)$ is the quotient field of $A$. }
 \end{notation}

   
   
 

Let $A$ be a domain with quotient field $Q(A)$, and let $t \in Q(A)$. Then $t$ is {\it almost integral} over $A$ if there is $0 \ne a \in A$ such that $at^n \in A$ for all $n>0$.  The domain $A$ is {\it completely integrally closed} if the only elements of $Q(A)$ that are almost integral over $A$ are those that are in $A$; equivalently, $\End(J) = A$ for all nonzero ideals $J$ of $A$.

\begin{lemma}  \label{cic} Let $X$ be an infinite  subset of $\Zar(F)$ such that ${{A}}(X)$ is a local ring.
  If ${{J}}(\lim X) \ne 0$ and $R$ is a completely integrally closed domain with ${{A}}(X) \subseteq R \subseteq F$, then 
   ${{A}}(X) \subseteq {{A}}(\lim X)  \subseteq  R. $ 

\end{lemma} 

\begin{proof} 
By Lemma~\ref{finally}, $J:={{J}}(\lim X)$ is a subset of $A$, hence of $R$, and so $JR$ is an ideal of $R$. By assumption, $J \ne 0$, so, since $R$ is completely integrally closed, $\End(JR) = R$. Thus $\End(J) \subseteq R$.  Since $J$ is an ideal of ${{A}}(\lim X)$, it follows that  ${{A}}(\lim X) \subseteq \End(J) \subseteq R$, which proves the lemma. 
\end{proof}

Thus, if the ring  ${{A}}(X)$ in Lemma~\ref{cic} is itself completely integrally closed, then ${{A}}(X) = {{A}}(\lim X)$. 
In the next theorem we carry this idea further to  show that a representation of a completely  integrally closed local subring $A$ of $F$ can be refined to a patch closed representation that is perfect as a topological space.   In the proof, we use that if $X$ is a patch closed subset of $\Zar(F)$ such that $A = {{A}}(X)$, then there is a {\it  minimal  patch closed representation} $Y$ of $A$ contained in $X$, in the sense that  $A = {{A}}(Y)$ and  $Y$ properly contains no  patch closed representation of $A$; see \cite[(4.2)]{OGraz}.  A minimal patch closed representation of $A$ need not be unique \cite[Example 4.3]{OGraz}. 
Recall that a topological space is {\it perfect} if every point in the space is a limit point.


\begin{theorem} \label{cic 2}
Let $X$ be a nonempty subset of $ \Zar(F)$ such that ${{J}}(X) \ne 0$.  If ${{A}}(X)$ is a completely integrally closed local ring that is not a valuation ring of $F$, then $\lim(X)$ contains  a  representation of ${{A}}(X)$ that  is perfect and closed in $\Zar(F)$ with respect to the patch topology.  
 \end{theorem}
 
 \begin{proof}  As discussed before the theorem,   the set of  patch closed representations of $A$ contains minimal elements with respect to set inclusion. Let $Y$ be a 
  minimal patch closed representation of $A$ contained in $\patch(X)$. 
  Since $\patch(X) = X \cup \lim(X)$,  Proposition~\ref{union} implies that  $$0 \: \ne \: {{J}}(X) \: =  {{J}}(\patch \: X)\: \subseteq \: {{J}}(Y).$$
  Since $A$ is not a valuation ring of $F$, $Y$ is infinite
  \cite[(11.11), p.~38]{N}.   Thus, since $A$ is completely integrally closed and $A = {{A}}(Y)$ with ${{J}}(Y) \ne 0$, 
  Lemma~\ref{cic} implies that $\lim(Y)$ is a patch closed representation of $A$. Since $Y$ is a minimal patch closed representation of $A$ and $\lim(Y) $ is a patch closed representation of $A$ contained in $Y$, we conclude that $Y = \lim(Y)$. Therefore, $Y$ is a perfect representation of $A$.
\end{proof}

Example~\ref{4.8} shows that the assumption in Theorem~\ref{cic 2} that ${{J}}(X)\ne 0$ is necessary. 

Restricting to the case where $F$ is a countable field, we obtain that the perfect representation given by the theorem is homeomorphic to the Cantor set:

\begin{corollary} \label{Cantor} Suppose the field $F$ is countable and $A$ is a completely integrally closed local subring of $F$  that  is not a valuation ring of $F$. If $X$ is a representation of $A$ such that ${{J}}(X) \ne 0$, then $\lim(X)$ contains a patch closed representation of $A$ that is homeomorphic in the patch topology to the Cantor set. 
\end{corollary}

\begin{proof}  By Theorem~\ref{cic 2}, 
$\lim(X)$ contains a patch closed representation $Y$ of $A$ that is perfect in the patch topology. Since $F$ is countable, the patch topology on $Y$ has a countable basis. As a patch closed subset of $\Zar(F)$,   $Y$ is quasicompact, so the fact that $Y$ has a countable basis implies the patch topology on  $Y$ is metrizable \cite[Theorem 23.1, p.~166]{W}. The patch topology on a spectral space has a basis of closed and open sets, so we conclude that $Y$ is a totally disconnected perfect metrizable space. Any such space is homeomorphic to the Cantor set 
\cite[Corollary 30.4, p.~217]{W}. 
\end{proof}

The next corollary shows that in the circumstances of Theorem~\ref{cic 2} any set of countably many valuation rings can be discarded from the representation $\lim(X)$ of $A$.  Thus $\lim(X)$ is uncountable in the setting of Theorem~\ref{cic 2}. 

\begin{corollary} \label{Baire} Let $\emptyset \ne X \subseteq \Zar(F)$ such that ${{J}}(X)\ne 0$.  If ${{A}}(X)$ is a completely integrally closed localring and not a  valuation ring of $F$, then ${{A}}(X) = {{A}}(Y)$ for every co-countable subset $Y$ of $\lim(X)$.

\end{corollary}


\begin{proof} Since $A:={{A}}(X)$ is not a valuation ring of $F$ and $A$ is the intersection of the valuation rings in $X$,   $X$ is infinite \cite[(11.11), p.~38]{N}. Let $\{V_i:i \in {\mathbb{N}}\}$ be a countable subset of $\lim(X)$. It suffices to show  that the set $\lim(X) \setminus \{V_i:i \in {\mathbb{N}}\}$ is a representation of $A$. 
 Theorem~\ref{cic 2} implies there is a representation $Y$ of $A$ such that  $Y \subseteq \lim(X)$ and $Y$ is perfect and closed in the patch topology.   For each $i$, define $Y_i = Y \setminus \{V_1,\ldots,V_i\}$.  Since $Y$ is Hausdorff in the patch topology, $Y_i$ is patch open in $Y$ for each $i$. Moreover, $Y_i$ is patch dense in $Y$ since the fact that $Y$ is perfect implies that none of the valuation rings $V_1,\ldots,V_i$ are patch isolated points in $Y$.  Therefore, the Baire Category Theorem \cite[Corollary~3.9.4, p.~198]{Eng} implies that as a countable intersection of dense open subsets of the quasicompact Hausdorff space $Y$,  the set $Y_{\infty}:=Y \setminus \{V_i:i \in {\mathbb{N}}\} = \bigcap_{i >0} Y_i$ is patch dense in $Y$. 
  Thus $Y = \patch(Y_\infty)$, and so Proposition~\ref{union} implies $A = {{A}}(Y) = {{A}}(Y_\infty)$. Since $Y_\infty \subseteq \lim(X) \setminus \{V_i:i\in {\mathbb{N}}\}$, it follows that  $\lim(X) \setminus \{V_i:i \in {\mathbb{N}}\}$ is a representation of $A$.  
\end{proof}


\begin{corollary} \label{late cor}  
Let $\emptyset \ne X \subseteq \Zar(F)$ be a set of rank one valuation rings such that ${{A}}(X)$ is a local ring with quotient field $F$ but ${{A}}(X)$ is not a valuation domain. If all but  at most  finitely many valuation rings in $X$ dominate $A$, then  $\lim(X)$  contains a  representation of ${{A}}(X)$ that is perfect and closed in $\Zar(F)$ in the patch topology.  
\end{corollary} 

\begin{proof} Since $X$ consists of rank one valuation rings, $A={{A}}(X)$ is a completely integrally closed domain. Also, since all but at most finitely many valuation rings in $X$ dominate $A$ and $A$ has quotient field $F$, it follows that  ${{J}}(X) \ne 0$. (Note that $A \ne F$, since $A$ is not a valuation ring.)  Theorem~\ref{cic 2} implies $\lim(X)$ contains a  representation of $A$ that is perfect and closed in the patch topology.  
\end{proof}

\begin{example} \label{4.8} {If the assumption in Corollary~\ref{late cor} that all but finitely many valuation rings in $X$ dominate $A$ is removed, the resulting statement is false. Let $A$ be a local Krull domain with Krull dimension $> 1$ and maximal ideal $M$. Let \begin{center}$X= \{A_P:P$ a height 1 prime ideal of $A\}$.\end{center} Then $X$ is a representation of $A$ consisting of DVRs. However, by Corollary~\ref{FC char}, $\lim(X)$ consists of a single valuation ring, namely, the quotient field of $A$. Moreover,   $A$ is completely integrally closed but  $\lim(X)$ is not a representation of $A$. This example also shows that the assumption that ${{J}}(X) \ne 0$ in Theorem~\ref{cic 2} is also necessary. 
 }
\end{example}

\begin{example} \label{vacant example}
 {This example shows that the restriction to rank one valuation rings in Corollary~\ref{late cor} is needed.  
%
 Let $D$ be a two-dimensional regular local ring with maximal ideal ${\ff m} = (x,y)D$, and let $t = x/y$. Then $U = D[t]_{{\ff m}D[t]}$ is a DVR and $U/{\ff M}_U \cong (D/{\ff m})(T)$, where $T$ is an indeterminate. Let $X$ be the collection of (rank two) valuation overrings of $D$ that dominate $D$ and are properly contained in $U$. Then
  $A:=D+{\ff M}_U$ is an integrally closed local domain with maximal ideal $M = {\ff M}_U$. Moreover, $A  = {{A}}(X)$ and 
   $\lim(X) =\{U\}$, and so $\lim(X)$ does not contain a perfect representation of $A$. In fact, the only valuation overrings of $A$ are those in $X \cup \{F\}$ (this follows for example from \cite[Theorem 4.1]{Fab}). Thus the only 
patch closed representation of $A$ is  $X \cup \{F\}$, and so $A$ does not have a perfect representation.  
 Where this example differs from those meeting the hypotheses of Corollary~\ref{late cor} is that $X$ does not consist of rank one valuation rings.   
}
\end{example}


Another variation on these ideas is given in Theorem~\ref{pd}, where the emphasis is on the residue fields of the representing valuation rings. 





\begin{theorem} \label{pd}
Let $A$ be a local integrally closed ring with maximal ideal $M$ and quotient field $F$, and let $X \subseteq \Zar(D)$ such that $A = A(X)$.
Suppose  $\End(M) = A$,  all but finitely many valuation rings in $X$ dominate $A$ and $A$ is not a valuation domain. 
 Then there is a  subset  $Z$ of $\lim(X)$ such that  $A = {{A}}(Z)$, $Z$ is perfect in the patch topology, and every valuation ring in $Z$ dominates $A$ and has residue field transcendental over the residue field of $A$.  
%
%
\end{theorem}

\begin{proof}
By Corollary~\ref{dominate}, each valuation ring in $\lim(X)$ dominates $A$. Therefore, by Lemma~\ref{finally}, $M = {{J}}(\lim X)$. Consequently, $M$ is an ideal of the ring ${{A}}(\lim X)$, and $$A = {{A}}(X) \subseteq {{A}}(\lim X) \subseteq \End(M) = A.$$  
Thus $\lim(X)$ is also a patch closed representation of $A$. 
 Let $Y$ be a minimal patch closed representation of $A$ contained in $\lim(X)$, and let $Z$ be the collection of valuation rings $V$ in $Y$ such that the residue field of $V$  is transcendental over the residue field of $A$.

 Since $A$ is not a valuation domain and $A$ is the intersection of the valuation rings in $Y$,   the set $Y$ is infinite \cite[(11.11), p.~38]{N}. By Lemma~\ref{trans}(2), $Z$ is nonempty. We claim first that ${{J}}(Z) \subseteq A$.  Suppose to the contrary that there exists $t \in {{J}}(Z) $ such that $t \not \in A$. If $t^{-1} \in A$, then $t^{-1} \in U$ for any $U \in Z$, a contradiction since $t \in {\ff M}_U$.  Thus $t \not \in A$ and $t^{-1} \not \in A$. 
 Since $Y$ is infinite, Lemma~\ref{trans}(2) implies there exists $U \in Z$ such that $t \not \in {\ff M}_U$, contrary to the choice of $t$. Therefore, ${{J}}(Z) \subseteq A$.

Next, since each $U \in Z$ dominates $A$, we have that $M = {{J}}(Z)$ is the maximal ideal of $A$. Since  $\End(M) =A$ and $M$ is an ideal of ${{A}}(Z)$, we have  $A = {{A}}(Z)$. 
 Consequently, $\patch(Z)$ is a representation of $A$ contained in the minimal patch closed representation $Y$ of $A$, which forces $Y  = \patch(Z)$.  
An argument similar to that at the beginning of the proof shows that $\lim(Z)$ is also a patch closed representation of $A$, so, again by the minimality of $Y$, we have $\lim(Z) = \patch(Z)$, which shows that $\patch(Z)$ is  perfect in the patch topology. 
That $Z$ is also perfect in the patch topology now follows from the fact that $Z$ is dense in $\patch(Z)$ with respect to the patch topology. 
\end{proof}



\begin{remark} {Example~\ref{vacant example} shows that the assumption that $\End(M) = A$ in Theorem~\ref{pd} is necessary, since in that example $A$ is an integrally closed domain that is not a valuation domain and $A$  does not have a perfect representation.} 
\end{remark}

\section{Rank one representations and Pr\"ufer intersections}

The proofs in this section and the next utilize  a topology that is dual to the Zariski topology on $\Zar(F)$.   
 The {\it inverse topology} on $\Zar(F)$ is the topology that has a  basis of closed sets the subsets of $\Zar(F)$ that are quasicompact and open in the Zariski topology; i.e., the closed sets are intersections of finite unions of sets of the form $\cal{U}(x_1,\ldots,x_n)$, $x_1,\ldots,x_n \in F$.  For more on the inverse topology in the context of the Zariski-Riemann space of a field, see \cite{FFL,OZR,OGraz}.

Recall from the introduction that an integral domain $A$ is a {\it Pr\"ufer domain} if each localization of $A$ at a maximal ideal of $A$ is a valuation domain. Viewing $\Zar(F)$ as a locally ringed space with structure sheaf ${{A}}$   given by ${{A}}({\cal U})$ for each nonempty Zariski open set ${\cal U}$ of $\Zar(F)$, it follows that for a subset $X$ of $\Zar(F)$, ${{A}}(X)$ is a Pr\"ufer domain with quotient field $F$ if and only if the closure of $X$ in the inverse topology of $\Zar(F)$  is an  affine scheme in $\Zar(F)$; i.e., ${{A}}(X)$ is  a Pr\"ufer domain if and only if  for $A = {{A}}(X)$  the closure of $X$ in the inverse topology of $\Zar(F)$ is  $\{A_P:P \in \Spec(A)\}$; see \cite{OZR}. Geometrical characterizations of subsets $X$ of $\Zar(F)$ such that  ${{A}}(X)$ is a Pr\"ufer domain with quotient field $F$ are given in \cite{ OGeo, OZR}. 

Our interest in this section is in topological conditions on $X$ that are sufficient to guarantee ${{A}}(X)$ is a Pr\"ufer domain.  In \cite{OCpt} it is shown that if both $X$ and $\lim(X)$ consist of rank one valuation rings of $F$, then ${{A}}(X)$ is a one-dimensional Pr\"ufer domain with nonzero Jacobson radical and quotient field $F$.  In this section, we also seek to control patch limit points, but not in terms of the rank of  valuation rings that occur as  limit points. Instead, we focus on the case in which $\lim(X)$ lies an affine subset of $\Zar(F)$; i.e., we require that ${{A}}(\lim X)$ is a Pr\"ufer domain with quotient field $F$.  In light of Proposition~\ref{union}, it is not surprising that this in itself is not enough to guarantee  ${{A}}(X)$ is a Pr\"ufer domain with quotient field $F$.   However, if we restrict $X$ to rank one valuation rings, then the assumption that ${{A}}(\lim X)$ is a Pr\"ufer domain with quotient field $F$ is sufficient for ${{A}}(X)$ also to be a Pr\"ufer domain under some straightforward assumptions on $X$. This is proved in Corollary~\ref{limit cor}. One of these restrictions is that ${{A}}(\lim X)$ is a {\it $G$-domain}; that is, a domain for which the intersection of nonzero prime ideals is nonzero. See \cite[Section 1.3]{Kap} for more on this class of rings. 

We use the following notation in relation to the inverse topology on $\Zar(F)$.

 \begin{notation} {
 For a subset $X$ of $\Zar(F)$, we denote by $\inv(X)$ the closure of $X$ in the inverse topology. Thus $\inv(X)$ is the intersection of the Zariski quasicompact open subsets of $\Zar(F)$ that contain $X$. The {\it generic closure} of $X$ in $\Zar(F)$ is denoted $\gen(X)$ and is defined to be the set of $W \in \Zar(F)$ such that $W \supseteq V$ for some $V \in X$. By \cite[Corollary, p.~45]{Hoc}, $\inv(X) = \gen(\patch \: X)$.   
 }
 \end{notation}

\begin{lemma} \label{tech lemma} Let $X_1$ be a set of rank one valuation rings in $\Zar(F)$, and let $X_2$ be a nonempty subset of $\Zar(F)$ such that ${{J}}(X_2) \ne 0$ and  $\lim(X_1) \subseteq X_2$. If $A:={{A}}(X_1) \cap {{A}}(X_2)$ is a local domain, then $A = {{A}}(X_2)$. 
%

\end{lemma}

\begin{proof} 
Let $X = X_1 \cup X_2$. 
If $X$ is finite, then, as an intersection of finitely many valuation rings of $F$,  $A$ is a valuation ring of $F$ \cite[(11.11), p.~38]{N}.
In this case 
 since $A$ is a valuation ring of $F$, $X_1$ consists of at most one valuation ring, necessarily of rank one.  Also, as overrings of a valuation domain, the valuation rings in  $X_2$ form a chain, one  that does not include $F$ since ${{J}}(X_2) \ne 0$. Thus every valuation ring in $X_2$ has rank at least one and so is contained in the only valuation ring in $X_1$ if $X_1$ is nonempty. 
It follows that if $X$ is finite, then $A = {{A}}(X_2)$
 and the claim is proved. 

Suppose next that $X$ is infinite. Since $A$ is local, Lemma~\ref{finally} implies ${{J}}(\lim X) \subseteq A$.  We claim that ${{J}}(\lim X) \ne 0$. 
 Since  $\lim(X_1) \subseteq X_2$, we have $$\lim(X) = \lim(X_1) \cup \lim(X_2) \subseteq   \lim(X_1) \cup \patch(X_2) = \patch(X_2).$$   Thus
${{J}}(X_2) = {{J}}(\patch \: X_2) \subseteq   {{J}}(\lim X)$ by Proposition~\ref{union}.
By assumption ${{J}}(X_2) \ne 0$, so 
%
%
 ${{J}}(\lim X) \ne 0$.

 Since ${{J}}(\lim X) \ne 0$ and $X_1$ consists of rank one valuation rings,  Lemma~\ref{cic} implies $ {{A}}(\lim X) \subseteq  
{{A}}(X_1).$   
Since also  $\lim(X) \subseteq \patch(X_2)$,  we have by Proposition~\ref{union} that  ${{A}}(X_2) = {{A}}(\patch(X_2)) \subseteq {{A}}(\lim X) \subseteq {{A}}(X_1)$.  It follows that $A = {{A}}(X_1) \cap {{A}}(X_2) = {{A}}(X_2)$.
\end{proof}




\begin{theorem}  \label{inv thm}
Let  $X_1$ be a nonempty set of rank one valuation rings in $\Zar(F)$, and  let $X_2 \subseteq \Zar(F)$ be such that  $\lim(X_1) \subseteq X_2$.
Then the following are equivalent. 
\begin{enumerate}[$(1)$]
\item ${{J}}(X_1) \ne 0$ and 
  ${{A}}(X_2)$ is a Pr\"ufer $G$-domain with quotient field $F$.
  
  \item  
 ${{A}}(X_1) \cap {{A}}(X_2)$ is a Pr\"ufer  $G$-domain with quotient field $F$. 
 \end{enumerate}
 \end{theorem}

 \begin{proof} (1) $\Rightarrow$ (2) Assuming (1), let $X = X_1 \cup X_2$, and let $A = {{A}}(X) = {{A}}(X_1) \cap {{A}}(X_2)$.
%
Without loss of generality we may replace $X_2$ with its inverse closure $\inv(X_2) = \gen(\patch(X_2))$. This is because ${{A}}(X_2) = {{A}}(\inv(X_2))$ by Proposition~\ref{union} so that working with $\inv(X_2)$ in the statement of the theorem is the same as working with $X_2$. 
First we prove that $A$ is a Pr\"ufer domain with quotient field $F$.  Let $M$ be a maximal ideal of $A$. We show that $A_M$ is a valuation ring of $F$.  
Since the inverse closed set  $X_2$ is patch closed and $\lim(X_1) \subseteq X_2$, we have  $$\patch(X) = \patch(X_1) \cup X_2 = (X_1 \cup \lim X_1) \cup X_2 = X_1 \cup X_2  =X.$$
Thus $X$ is patch closed, and  
 $$\inv(X) = \gen(\patch(X)) = \gen(X) = \gen(X_1) \cup \gen(X_2) = X_1 \cup X_2 = X,$$ where the second to last equality follows from the fact that $X_1$ consists of rank one valuation rings and $X_2$ is inverse closed. Consequently, $X$ is inverse closed. 
 
 
 
 Since $X$ is inverse closed,   there exists an inverse closed subset $Y$ of $X$   such that $A_M = {{A}}(Y)$ \cite[Corollary 5.7]{OZR}.  Write $Y_1 = Y \cap X_1$ and $Y_2 = (Y \cap X_2) \setminus \{F\}$.  
 Then $Y = Y_1 \cup Y_2 \cup \{F\}$. 
 If $Y$ is finite, then $A_M$ is an intersection of finitely many valuation rings of $F$ and hence is a valuation ring of $F$ \cite[(11.11), p.~38]{N}. Thus if $Y$ is finite the proof that $A_M$ is a valuation ring of $F$ is complete. We assume for the rest of the argument that $Y$ is infinite.

With the aim of applying Lemma~\ref{tech lemma}, we make several observations in order to show that $Y_1$ and $Y_2$ satisfy the hypotheses of the lemma.  Specifically, we show $\lim(Y_1) \subseteq Y_2$ and $Y_2$ is  a nonempty set with ${{J}}(Y_2) \ne 0$.  

First, to see that $\lim(Y_1) \subseteq Y_2$, observe that 
since $X_2$ is patch closed and $\lim(X_1) \subseteq X_2$, we have $\lim(Y) \subseteq \lim(X) = \lim(X_1) \cup X_2 = X_2.$  Thus, since $Y$ is inverse closed and hence patch closed, we have $\lim(Y_1) \subseteq X_2 \cap Y$. If $F \in \lim(Y_1)$, then, since $Y_1 \subseteq X_1$, we have $F \in \lim(X_1)$. In this case, Proposition~\ref{union} implies that ${{J}}(X_1) = 0$,
contrary to assumption. Thus $\lim(Y_1) \subseteq (Y \cap X_2) \setminus \{F\} = Y_2$.

Next we show that $Y_2$ is  a nonempty set with ${{J}}(Y_2) \ne 0$. If $Y_2$ is empty, then the subset $\lim(Y_1)$ of $Y_2$ is empty, and
so $Y_1$ is finite 
 by Proposition~\ref{union}. In this case, $Y = Y_1 \cup Y_2 \cup \{F\}$ is a finite set, contrary to assumption. Therefore, $Y_2$ is nonempty. 
 Furthermore, since $Y_2 \subseteq X_2 \setminus \{F\}$, we have ${{J}}(X_2 \setminus \{F\}) \subseteq {{J}}(Y_2)$. 
 Since ${{A}}(X_2)$ has quotient field $F$,  each valuation ring in $X_2 \setminus \{F\}$ is centered on a nonmaximal prime ideal of the $G$-domain ${{A}}(X_2)$, so necessarily ${{J}}(X_2 \setminus F) \ne 0$. Thus ${{J}}(Y_2) \ne 0$.   
  
  Having shown that $Y_1$ and $Y_2$ satisfy the hypotheses of Lemma~\ref{tech lemma} (with $Y_1$ and $Y_2$ playing the roles in the lemma of ``$X_1$'' and ``$X_2$,'' respectively), we obtain that 
   $A_M =  {{A}}(Y_2)$. 
 Since $Y_2 \subseteq  X_2$, we have ${{A}}(X_2) \subseteq {{A}}(Y_2) = A_M$.  
 Since ${{A}}(X_2)$ is a Pr\"ufer domain with quotient field $F$, we conclude that as a local overring of a Pr\"ufer domain with quotient field $F$, $A_M$  
  is a valuation ring of $F$. This proves that $A$ is a Pr\"ufer domain with quotient field $F$.  

To see finally that $A$ is a $G$-domain, we first recall that since $A$ is a Pr\"ufer domain and $X=X_1 \cup X_2$ is an inverse closed representation of $A$, then $X$ is the set of valuation rings between $A$ and $F$, each of which is a localization of $A$ at a prime ideal \cite[Proposition 5.6(5)]{OZR}. Similarly, $X_2$ is the set of valuation overrings of ${{A}}(X_2)$, each of which is a localization of ${{A}}(X_2)$ at a prime ideal. Let $I$ be the intersection of the nonzero prime ideals of ${{A}}(X_2)$. Since $X_1$ consists of rank one valuation rings and $X = X_1 \cup X_2$ is the set of localizations of $A$ at prime ideals, it follows that ${{J}}(X_1) \cap I$ is the intersection of nonzero prime ideals of $A$.  If ${{J}}(X_1) \cap I = 0$, then since $A$ has quotient field $F$ we have  ${{J}}(X_1) = 0$ or $I = 0$. By assumption ${{J}}(X_1) \ne 0$ and, since ${{A}}(X_2)$ is a $G$-domain, $I \ne 0$. Thus ${{J}}(X_1) \cap I \ne 0$, and hence $A$ is a $G$-domain.  

(2) $\Rightarrow$ (1)
Since $A={{A}}(X_1) \cap {{A}}(X_2)$ has quotient field $F$,  the valuation rings in $X_1$ are centered on  nonmaximal  prime ideals of $A$. Since $A$ is a $G$-domain, we conclude that ${{J}}(X_1) \ne 0$.  Moreover, since $A$ is a Pr\"ufer domain with quotient field $F$, each prime ideal of ${{A}}(X_2)$ is extended from a prime ideal of $A$ \cite[Theorem 6.10]{LM}. Thus, since $A$ is a $G$-domain, so is ${{A}}(X_2)$.  Finally, every ring between a Pr\"ufer domain and its quotient field is a Pr\"ufer domain \cite[Theorem 6.10]{LM}, so ${{A}}(X_2)$ is a Pr\"ufer domain with quotient field $F$.  
\end{proof}

 \begin{corollary}  \label{limit cor} Let $X$ be a nonempty set of rank one valuation rings in $\Zar(F)$.  Then ${{A}}(X)$ is a Pr\"ufer $G$-domain with quotient field $F$ if and only if 
  ${{J}}(X) \ne 0$  
and ${{A}}(\lim X)$ is a Pr\"ufer $G$-domain with quotient field $F$. \end{corollary}

 \begin{proof} 
 Let $X_2 = \lim(X)$. 
By
  Theorem~\ref{inv thm}, 
     ${{A}}(X) \cap {{A}}(X_2)$ is a Pr\"ufer $G$-domain with quotient field $F$ if and only if  ${{J}}(X) \ne 0$ and 
     ${{A}}(X_2)$ is a Pr\"ufer $G$-domain with quotient field $F$.  
              By Proposition~\ref{union}, ${{A}}(X) \subseteq  {{A}}(\lim X) = {{A}}(X_2)$, so 
             the corollary  follows. 
 \end{proof}

 
 

\begin{corollary}  \label{H cor} Let $V_1,\ldots,V_n$ be rank one valuation rings of $F$. 
If $B$ is a Pr\"ufer $G$-domain with quotient field $F$, then $V_1 \cap \cdots \cap V_n \cap B$ is a Pr\"ufer $G$-domain with quotient field $F$. 
\end{corollary}

\begin{proof} Let $X_1 = \{V_1,\ldots,V_n\}$, and let $X_2 = \{B_M:M \in \Max(B) \}$. Since $B$ is a Pr\"ufer domain with quotient field $F$, $X_2$ is a subset of $\Zar(F)$.  
Now $\lim(X_1) = \emptyset \subseteq X_2$, ${{J}}(X_1) \ne 0$ and $B = {{A}}(X_2)$ is a Pr\"ufer $G$-domain with quotient field $F$. By Theorem~\ref{inv thm}, $V_1 \cap \cdots \cap V_2 \cap B = {{A}}(X_1) \cap {{A}}(X_2)$ is a Pr\"ufer $G$-domain with quotient field $F$. 
\end{proof} 

The next example shows the importance of restricting to rank one valuation rings in Corollary~\ref{H cor}. 

\begin{example}
In general, the intersection of a valuation ring of $F$ of  rank more than~$1$ and a  Pr\"ufer $G$-domain need not be a Pr\"ufer domain. For example, let $A$ be as in Example~\ref{vacant example}. With the notation of the example, $B = D[t] + {\ff M}_U$ is a two-dimensional Pr\"ufer domain for which ${\ff M}_U$ is the unique nonzero nonmaximal prime ideal of $B$. Hence $B$ is a $G$-domain. Let $V = D[1/t]_{({\ff m},1/t)}+{\ff M}_U$. Then $V$ is a rank two valuation ring of $F$ contained in $U$ and for which $V \cap B = A$.  Since $A$ is a local domain that is not a valuation domain,  $A$ is not a Pr\"ufer domain. \end{example}

\section{Rank one representations having few limit points}

As another application of the ideas in Section 3 we classify in this section the subrings $A$ of the field $F$ for which there exists a set $X$ of rank one valuation rings such that $A = A(X)$ and $X$ has 
finitely many patch limit points. By Corollary~\ref{FC char}, such rings include the Krull domains, and more generally the {\it finite real character domains}, where an integral domain $A$ with quotient field $F$ has {\it finite real character} if there is a finite character collection $X$ of rank one valuation rings in $\Zar(F)$  such that $A = {{A}}(X)$. If in addition each valuation ring in $X$ is a localization of $A$, then $A$ is a {\it generalized Krull domain}. In order for a finite real character domain to be a generalized Krull domain, it is sufficient that the valuation rings in $X$ have rational rank $1$ (see \cite{Krull} or \cite[Corollary 5.2]{Ohm}).  Generalized Krull domains and finite real character domains have been well studied; see for example \cite{G, Griffin, Griffin2, HOFC, Ohm, Pirtle}. 

Our classification of the case where $X$ has finitely many patch limit points hinges as usual on whether ${{J}}(X) = 0$.  
We show in Theorem~\ref{finite lim case} that if $\lim(X)$ is finite and ${{J}}(X) \ne 0$, then ${{A}}(X)$ is a Pr\"ufer domain.
  The classification in Corollary~\ref{classify}  of the case in which ${{J}}(X)$ is not necessarily the zero ideal  involves Pr\"ufer domains, finite real character domains and an intersection of 
both types of rings.

  Theorem~\ref{cardinal} is our main technical result for dealing with the case in which  $\lim(X)$ is finite. 
It is proved in a stronger form than what is needed in what follows. 
Theorem~\ref{cardinal} requires the following cardinality result involving Pr\"ufer intersections of valuation rings.



\begin{lemma} \label{Nagata} 
Let $X$ be a nonempty subset of $\Zar(F)$. 
If ${{A}}(X)$ contains a local subring  whose residue field has cardinality $\alpha$ and $|X| < \aleph_0 \cdot \alpha$,  then ${{A}}(X) $ is a Pr\"ufer domain with quotient field $F$. 
\end{lemma}

\begin{proof} If $X$ is a finite set, then ${{A}}(X)$ is a Pr\"ufer domain with quotient field $F$ by    \cite[(11.11), p.~38]{N}. Otherwise, suppose $X$ is infinite. Since $|X| < \aleph_0 \cdot \alpha$, this forces $\alpha$ to be infinite, and hence $\aleph_0 \cdot \alpha = \alpha$. Then we have $|X| < \alpha$. The fact that ${{A}}(X)$ is a Pr\"ufer domain with quotient field $F$ 
now follows from \cite[Corollary 3.8]{OGeo} (which 
 extends \cite[Theorem 6.6]{OR}). \end{proof}

In the next theorem we use the notation $\min(\lim X)$ for the minimal elements of $\lim(X)$ with respect to set inclusion; see Remark~\ref{min notation}.

\begin{theorem} \label{cardinal}
Let   $X_1$ be a nonempty set of rank one valuation rings in $\Zar(F)$ such that ${{J}}(X_1) \ne 0$,  let $X_2$ be a patch closed subset of $\Zar(F)$ with $\min(\lim X_1) \subseteq X_2$, and let $ A={{A}}(X_1) \cap {{A}}(X_2)$.
 If  $\min X_2$ is finite or for each maximal ideal $M$ of $A$, $A_M$ contains a local ring whose residue field has cardinality greater than $ |\min X_2|$, then 
 $A$ is a Pr\"ufer domain with  quotient field $F$.   
\end{theorem}

\begin{proof}  
Let $X = X_1 \cup X_2$.   Since $X_2$ is patch closed, we have $\inv(X_2) = \gen(X_2)$. 
Since $\inv(X)  = \gen(\patch(X))$ and $\patch(X_1) = X_1 \cup \lim(X_1)$, we have  
\begin{eqnarray*} \inv(X) 
 & = &  \gen(\patch(X_1)) \cup \inv(X_2) \\
& =  &  \gen(X_1) \cup \gen(\lim X_1) \cup \gen(X_2) \\
 &= & X_1 \cup \gen(\min(\lim X_1))) \cup \gen(X_2) \\
 &= &  X_1 \cup \gen(X_2),
 \end{eqnarray*}
 where the third equality follows from the fact that $X_1$ consists of rank one valuation rings and the fourth from the fact that $\min(\lim X_1) \subseteq X_2.$
 %
 %

 To prove that $A$ is a Pr\"ufer domain with quotient field $F$, 
 it suffices to show that $A_M$ is a valuation ring of $F$ for each maximal ideal $M$ of $A$. 
 Let $M$ be a maximal ideal of $A$. 
Then $A_M$ is the intersection of the valuation rings in the patch closed set $Y=\{V \in \inv(X):A_P \subseteq V\}$ \cite[Corollary 5.7]{OZR}. 

  Since $Y$ is patch closed, $Y$ contains minimal elements (see Remark~\ref{min notation}) and we can  replace $Y$ with $\min Y$. After such a replacement, we have a subset $Y$ of $\Zar(F)$ with the following properties. 
\begin{itemize}
\item[(i)]  $Y \subseteq \inv(X) = X _1\cup \gen(X_2)$,
\item[(ii)]  $Y$ consists of incomparable valuation rings, and 
\item[(iii)]  $A_M = {{A}}(Y)$.  
\end{itemize}

 


 If $Y$ is finite, then $A_M$ is a valuation ring of $F$ \cite[(11.11), p.~38]{N} 
  and the claim is proved. Thus 
  we assume for the rest of the proof that $Y$ is infinite. 
  
  


  \smallskip
 
{\textsc{Claim 1.}}  Any collection of incomparable valuation rings in $\gen(X_2)$ has cardinality at most $ |\min X_2|$. 
 
 \smallskip
 
 {\it Proof of Claim 1}.   Let $Z = \{W_j:j \in J\}$ be a set of incomparable distinct valuation rings in $\gen(X_2)$. Since $X_2$ is patch closed, $\gen(X_2) = \gen(\min X_2)$.  For each $j \in J$, there is $V_j \in \min X_2$ such that $V_j \subseteq W_j$.  If $j,k \in J$ such that $V_j = V_k$, then $V_j \subseteq W_j \cap W_k$, so that, since $V_j$ is a valuation ring with quotient field $F$, it follows that $W_j \subseteq W_k$ or $W_k \subseteq W_j$.  Since $Z$ consists of incomparable valuation rings, it follows that $W_j = W_k$, and hence $j = k$. Therefore, $|Z|= |J| \leq |\min X_2|$. 
 This proves Claim 1.  
 \qed
 
 \medskip
 
 Since $Y \subseteq X _1\cup \gen(X_2)$, we have
 $Y  = (Y \cap X_1) \cup (Y \cap \gen(X_2)).$  
 Write $$Y \cap \gen(X_2) = \{U_i:i \in I\}.$$ Then $$Y = (Y \cap X_1) \cup \{U_i:i \in I\}.$$

 \smallskip
 
{\textsc{Claim 2.}} 
  $|I| \:  \leq \: |\min X_2|$. 
  
  \smallskip
  
{\it Proof of  Claim 2}.  Since $Y$ consists of incomparable valuation rings and the set $\{U_i:i \in I\} = Y \cap \gen(X_2)$, this follows from Claim 1. 
%
%
  \qed
  
 \smallskip
  
  As in the proof of Lemma~\ref{finally},
 for each $i$  there exists  $V_i\in \Zar(F)$  such that $V_i$ dominates $A_M$ and $V_i \subseteq U_i$ .  Since the members of $\{U_i:i \in I\}$  are incomparable, 
we have as in the proof of Claim 1 that $V_i = V_j$ if and only if $i = j$.   
 Thus the members of $\{V_i: i \in I\}$ are incomparable  by Claim~1, and we  have by Claim 2 that
  $$|\{V_i:i \in I\}| = |I| \leq |\min X_2|.$$  Let
 $Z_1 = Y \cap X_1$, $Z_2 = \lim(Z_1) \cup \{V_i:i \in I\}$ and   
   $Z=Z_1 \cup Z_2$.

  \smallskip
  
  
  
    

    {\textsc{Claim 3.}} The ring ${{A}}(Z_2)$ is a Pr\"ufer domain with quotient field $F$.   
  \smallskip
  
 {\it Proof of Claim 3}.    
 Since $Z_2 = \lim(Z_1) \cup \{V_i:i \in I\}$, we have  $${{A}}(Z_2) = {{A}}(\lim Z_1) \cap {{A}}(\{V_i\}) = {{A}}(\min(\lim Z_1)) \cap {{A}}(\{V_i\}).$$

 The set $\min(\lim Z_1)$ consists of incomparable valuation rings. Also, since $X_2$ is patch closed, we have $$\lim(Z_1) \subseteq \lim(X) = \lim(X_1) \cup \lim(X_2)  \subseteq \lim(X_1) \cup \gen(X_2).$$ Therefore, since by assumption
$\min(\lim X_1) \subseteq X_2$, we have  
  $$\min(\lim Z_1) \subseteq \gen(\min(\lim X_1)) \cup  \gen(X_2) = \gen(X_2).$$  
Thus, by Claim 1, $|\min(\lim Z_1)|   \leq 
|\min X_2|$. By Claim 2, 
 $|I| \leq |\min X_2|$.  Thus ${{A}}(Z_2)$ can be written as an intersection of $ \leq 2\cdot |\min X_2|$ valuation rings in $\Zar(F)$. Since every valuation ring in $Z$ contains $A_M$ and  $A_M$ contains a local ring whose residue field has cardinality greater than $|\min X_2|$, Lemma~\ref{Nagata} implies that ${{A}}(Z_2)$ is a Pr\"ufer domain with quotient field $F$.  
 %
\qed
 
 \smallskip

  {\textsc{Claim 4.}}  $A_M$ is a valuation ring of $F$.  Hence $A$ is a Pr\"ufer domain with  quotient field $F$.  

  \smallskip
  
   {\it Proof of Claim 4}.  
By the choice of $Y$, we have  $A_M   ={{A}}(Y)$. As noted before Claim 2, we have $Y = Z_1 \cup \{U_i:i \in I\}$, so that $A_M = {{A}}(Z_1) \cap {{A}}(\{U_i\})$. Recall that  $Z_2 = \lim(Z_1) \cup \{V_i:i \in I\}$.
By the choice of the $V_i$, we conclude that $$A_M \subseteq {{A}}(Z_1) \cap {{A}}(Z_2) \subseteq {{A}}(Z_1)  \cap {{A}}(\{V_i\}) \subseteq {{A}}(Z_1) \cap {{A}}(\{U_i\}) = A_M.$$
Thus $A_M = {{A}}(Z_1)  \cap {{A}}(Z_2)$.   

   To prove $A_M  = {{A}}(Z_1) \cap {{A}}(Z_2)$ is a valuation ring of $F$,  Lemma~\ref{tech lemma} and Claim 3 imply it suffices to verify that
  $Z_2$ is a nonempty set and  
    ${{J}}(Z_2) \ne 0$. If $Z_2$ is empty, then $\lim(Z_1)$ is empty and $Z_1$ is finite by Proposition~\ref{union}.  
Moreover, in this case $I = \emptyset$, so that $Y = Z_1 \cup \{U_i:i \in I\} = Z_1$ is finite, contrary to the assumption that $Y$ is infinite. Therefore, $Z_2$ is a nonempty set.

 Next we claim that   ${{J}}(Z_2) \ne 0$. Since $Z_2 = \lim(Z_1) \cup \{V_i:i \in I\}$, it suffices to show that ${{J}}(\lim Z_1) \cap {{J}}(\{V_i\}) \ne 0$.  
      
      \smallskip 
      
      {\textsc{Claim}} 4a: 
If $\lim(Z_1) \ne \emptyset$, then ${{J}}(\lim Z_1) \ne 0$.   

\smallskip

{\it Proof of Claim 4a.}
 By assumption, ${{J}}(X_1) \ne 0$. Since $Z_1 \subseteq X_1$, we have by Proposition~\ref{union} that $0 \ne {{J}}(X_1)       
    \subseteq {{J}}(Z_1) \subseteq {{J}}(\lim Z_1)$. \qed
    
    \smallskip
        
    {\textsc{Claim}} 4b:  If $I \ne \emptyset$, then ${{J}}(\{V_i\}) \ne 0$.   
    
    \smallskip
    
    {\it Proof of Claim 4b.}  If  ${{J}}(\{V_i\}) = 0$, then since each $V_i$ dominates $A_M$ we have $M = 0$ and $A = F$. By assumption, $X_1$ is nonempty, so $A \ne F$ since $A$ is contained in a rank one valuation ring of $F$.  This  implies ${{J}}(\{V_i\}) \ne 0$.  \qed

    \smallskip
    
      {\textsc{Claim}} 4c:  If $I = \emptyset$, then ${{J}}(Z_2) \ne 0$. 
     
     \smallskip
     
     {\it Proof of Claim 4c.} If $I = \emptyset$, then
$Z_2 = \lim(Z_1)$.       
      Since $Z_2$ is nonempty, we have then by Claim 4a that ${{J}}(Z_2) \ne 0$. \qed
         
     \smallskip
     
     {\textsc{Claim}} 4d:  If $\lim(Z_1) = \emptyset$, then ${{J}}(Z_2) \ne 0$ . 
     
     \smallskip
     
     {\it Proof of Claim 4d.} If $\lim(Z_1) = \emptyset$, then
$Z_2 = \{V_i:i \in I\}$.       
      Since $Z_2$ is nonempty, we have then that $I \ne \emptyset$ and hence, by Claim 4b, ${{J}}(Z_2) \ne 0$. \qed
     
     \smallskip
     
      {\textsc{Claim}} 4e:  If $I \ne \emptyset$ and $\lim(Z_1) \ne \emptyset$, then ${{J}}(Z_2) \ne 0$.  
      
      \smallskip
      
      {\it Proof of Claim 4e.} In this case, we have by Claims 4a and 4b that ${{J}}(\lim Z_1) \ne 0$ and ${{J}}(\{V_i\}) \ne 0$.  By Claim 3, ${{A}}(Z_2)$ has quotient field $F$. As the intersection of two nonzero ${{A}}(Z_2)$-submodules of $F$, the ideal ${{J}}(Z_2) = {{J}}(\lim Z_1) \cap {{J}}(\{V_i\})$ is nonzero. \qed 
      
      \smallskip
      
      Returning to the proof of Claim 4, we have established in Claims 4c, 4d and 4e that ${{J}}(Z_2) \ne 0$. By Lemma~\ref{tech lemma}, $A_M$ is a valuation ring of $F$.  
     This proves Claim 4 and completes the proof of the theorem.
\end{proof}

In the theorem, $\min(\lim X_1)$ is  itself   an interesting choice for $X_2$. For example, it follows from the theorem that if $F$ contains a field $k$ and $X$ is a set of rank one valuation rings of $F$ containing $k$ with $|\min(\lim X)| < |k|$, then ${{A}}(X)$ is a Pr\"ufer domain with quotient field $F$.  Our focus next is on the case in which $\min(\lim X)$ is finite. In this case there is no need of the presence of a subfield of ${{A}}(X)$ of sufficiently large cardinality.

\begin{theorem} \label{finite lim case}
Let $A$ be a proper subring of $F$.  Then the following are equivalent.
\begin{enumerate}[$(1)$]
\item There is $X \subseteq \Zar(F)$ such that  $A = {{A}}(X)$, ${{J}}(X) \ne 0$,   $\min(\lim X)$ is finite, and 
 all but  finitely many valuation rings in $X$ have rank one. 
 
 \item 
  $A$ is a Pr\"ufer domain with nonzero Jacobson radical and  quotient field $F$ such that  all but finitely many maximal ideals of $A$ have  height one and are the radical of a finitely generated ideal. 
 \end{enumerate}
 
   \end{theorem} 
 
 \begin{proof} (1) $\Rightarrow$ (2)
 Let $X_1$ be the set of rank one valuation rings in $X$. 
 If $X_1$ is empty, then $X$ is finite, and hence $A$, as an intersection of finitely many valuation rings of $F$, is a Pr\"ufer domain with quotient field $F$ \cite[(11.11), p.~38]{N}.   Suppose $X_1$ is not empty, and 
  let $X_2 = (X \setminus X_1) \cup \lim(X_1)$.  Then ${{J}}(X_1) \ne 0$ since  ${{J}}(X) \ne 0$. Also, 
since $X \setminus X_1$ is finite and $\lim(X_1)$ is patch closed, we have that  
  $X_2$ is a patch closed subspace of $\Zar(F)$ with $\min(\lim X_1) \subseteq X_2$.  Since  $\min(\lim X_1)$ and $X \setminus X_1$ are finite, it follows that $\min X_2 $ is finite. 
 By Theorem~\ref{cardinal},   $A$ is a Pr\"ufer domain with  quotient field $F$.

 Let $ Y = \{A_M:M \in \Max(A)\}$.  
   We claim  that $Y \subseteq \patch(X)$. Since $\inv(X) = \gen(\patch(X))$ and $\inv(X)$ is the set of valuation rings between the Pr\"ufer domain $A$ and its quotient field $F$ \cite[Proposition 5.6(5)]{OZR}, it follows that every minimal valuation overring of $A$ is in $\patch(X)$. Thus, since $Y$ is the set of minimal valuation overrings of the Pr\"ufer domain $A$, we have $Y \subseteq \patch(X)$. 
   
   We now use the fact that $Y \subseteq \patch(X)$ to verify that $A$ has nonzero Jacobson radical. 
   By Proposition~\ref{union}, the fact that $Y \subseteq \patch(X)$ implies  ${{J}}(X) = {{J}}(\patch(X)) \subseteq {{J}}(Y)$, so that ${{J}}(Y) \ne 0$.  By the choice of $Y$,   $J(Y)$ is the Jacobson radical of $A$, so we have proved that $A$ has nonzero Jacobson radical.

It remains to show that  all but finitely many maximal ideals of $A$ have  height one and are the radical of a finitely generated ideal.
Let $Y_1$ be the set of rank one valuation rings in $Y$ that are patch isolated points in $Y$, and let $Y_2$ be the set of valuation rings in $Y$ that have rank $>1$ and are patch isolated in $Y$. Observe that 
 $$Y =  Y_1 \cup Y_2 \cup (Y \cap \lim Y).$$  We claim that $Y_2$ and $Y \cap \lim(Y)$ are finite sets. First, 
 since the valuation rings in $Y$ are minimal over $A$ and $\lim(Y) \subseteq \lim(\patch(X)) = \lim(X)$, we have 
 $$Y \cap \lim(Y) \subseteq Y \cap \lim(X) \subseteq \min(\lim X).$$ Hence $Y \cap \lim(Y)$ is finite since by assumption $\min(\lim X)$ is finite.  
 
 Next, to see that $Y_2$ is finite, observe that
 since $Y_2 \subseteq Y \subseteq \patch(X) = X \cup \lim(X)$, we have  for each $V \in Y_2$ that $V \in X$ or $V \in \lim(X)$. By assumption there are only finitely many valuation rings in $X$ of rank $>1$. On the other hand, if $V \in Y_2 \cap \lim(X)$, then  since $V$ is minimal as a valuation overring of $A$, we have $V \in \min(\lim X)$, a finite set. It follows that $Y_2 $ is a finite set. We have proved  that $Y_2 \cup  (Y \cap \lim Y)$ is a finite set. Therefore, $Y \setminus Y_1$ is a finite set. 
 
Since $Y \setminus Y_1$ is finite, to complete the verification of (2) it suffices to show that each maximal ideal $M$ of $A$ such that $A_M \in Y_1$ has height one and is the radical of a finitely generated ideal of $A$.  It is clear that $M$ has height $1$ since the valuation ring $A_M$, as a member of $Y_1$, has rank one. Also, as a member of $Y_1$, $A_M$ is a patch isolated point in $Y$.  Since the valuation rings in $Y$ are minimal over $A$, the patch and inverse topologies agree on $Y$  \cite[Corollary 2.6]{ST}. (To apply the cited reference we are using here that $\patch(X)$ is a spectral space with respect to the inverse topology \cite[Proposition 8]{Hoc} and that the specialization order in this topology on $\Zar(F)$ is set inclusion.)  Thus $A_M$ is isolated in the inverse topology of $Y$. From this it follows that there is a finitely generated ideal $I$ of $A$ contained in $M$ but in no other maximal ideal of $A$ \cite[Theorem 6.4]{OGraz}. Since $M$ has height one, $M$ is thus the radical of $I$.

(2) $\Rightarrow$ (1) Assuming (2), let $X = \{A_M:M \in \Max(A)\}$. Since $A$ is a Pr\"ufer domain with quotient field $F$, $X \subseteq \Zar(F)$.  By \cite[Theorem 6.4]{OGraz}, the points in $X$ that are isolated in the inverse topology correspond to the maximal ideals $M$ of $A$ such that $M$ contains a finitely generated ideal not contained in any other maximal ideal.  The subspace $\{A_P:P \in \Spec(A)\}$ is a spectral subspace of $\Zar(F)$, so the inverse and patch topologies coincide on its minimal elements with respect to set inclusion \cite[Corollary 2.6]{ST}. Therefore, since all but at most finitely many maximal ideals in $A$ have height one and are the radical of a finitely generated ideal in $A$,  it follows that all but finitely many valuation rings in $X$ are isolated in the patch topology and have rank one. Consequently, $X$ contains at most finitely valuation rings of rank $>1$, and $X$ contains at most finitely many patch limit points. 

We claim that $\min(\lim X)$ is finite. If $V \in X \cap \min(\lim X)$, then there are only finitely many possibilities for $V$ since there are only finitely many patch limits in $X$. Thus $X \cap \min(\lim X)$ is finite. It remains to show that $\min(\lim X) \setminus X$ is finite. 
 Let $V \in \min(\lim X) \setminus X$. Because of how we have defined $X$, $V$ is not a minimal valuation overring of $A$. Hence there is a valuation ring $U \in X$ such that $U \subsetneq V$.  If $V = F$, then $F \in \min(\lim X)$, so that $\lim(X) = \{F\}$.  But then Proposition~\ref{union} implies 
 ${{J}}(X)= 0$. However, by the choice of $X$, ${{J}}(X)$ is the Jacobson radical of $A$,  which we have assumed to be nonzero. Consequently, $V \ne F$, and, since $U \subsetneq V$, $U$ has rank $>1$.  By assumption, there are only finitely many localizations of $A$ at a maximal ideal that have rank $>1$. Thus there are only finitely many possibilities for $U$.  Since each of the valuation rings in the set  $\min(\lim X) \setminus X$ of incomparable valuation rings contains one of these rank $>1$ valuation rings, it follows that $\min(\lim X) \setminus X$ is also finite. Thus $\min(\lim X) = (X \cap \min(\lim X)) \cup (\min(\lim X) \setminus X)$ is finite. 
 \end{proof}

\begin{corollary} \label{classify} Let  $X \subseteq \Zar(F)$ such that ${{A}}(X)$ is a proper subring of $F$ with quotient field $F$.  If $\lim(X)$ is finite and all but finitely many valuation rings in $X$ have rank one, then 
the  ring ${{A}}(X)$ is one of the following types of rings:
\begin{enumerate}[{\em (a)}] 
\item a domain of finite real character, 
\item a  Pr\"ufer domain with nonzero Jacobson radical such that all but finitely many maximal ideals  have height one and are the radical of a finitely generated ideal, or
\item an intersection ${{A}}(X) = B \cap C$, where $B$ is as in (a) and $C$ is as in (b). 

\end{enumerate}

\end{corollary}

\begin{proof} 
Suppose first  that $F \not \in \lim(X)$.  Since $A(X)$ has quotient field $F$, Corollary~\ref{F cor 2} implies $J(X) \ne 0$. 
By Theorem~\ref{finite lim case}, ${{A}}(X)$ is a Pr\"ufer domain meeting the requirements in (b).

 Now suppose $F \in \lim(X)$.        
Since $\lim(X)$ is finite and the patch topology on $\Zar(F)$ is Hausdorff with a basis of closed and open sets, there is a patch closed and open subset $Z$ of $X$ such that  $\lim(Z) = \{F\}$ and every valuation ring in $Z \setminus \{F\}$  has rank one.  Let $Y$ be the patch closed subset of $X$ given by $Y = X \setminus Z$.  
 Since $\lim(Z)=\{F\}$ and every valuation ring in $Z \setminus \{F\}$ has rank one, Corollary~\ref{FC char} implies ${{A}}(Z)$ is a domain of finite real character. If $Y$ is empty, then ${{A}}(X) = {{A}}(Z)$ and we are in case (a). 
 
 On the other hand, suppose $Y$ is nonempty. Since $Y$ is  a patch closed set not containing $F$, 
we have 
 $F \not \in \lim(Y)$. By Corollary~\ref{F cor 2}, $J(Y) \ne 0$. Also, all but finitely many valuation rings in $Y$ have rank 1, so  Theorem~\ref{finite lim case} implies that ${{A}}(Y)$ is a Pr\"ufer domain meeting the requirements of (b). Since $X = Y \cup Z$, we have ${{A}}(X)  = {{A}}(Y) \cap {{A}}(Z)$, so ${{A}}(X)$ is described as in (c). \end{proof}

We illustrate Theorem~\ref{finite lim case} with an application to the theory of quadratic transforms of regular local rings.  
 Let $R$ be a regular local ring with maximal ideal  ${\ff m}$.  Let $x \in {\ff m} \setminus {\ff m}^2$, and let $P$ be a prime ideal of the ring $R[{\ff m}x^{-1}]$ that contains ${\ff m}$. Then the ring $R_1 = R[{\ff m}x^{-1}]_P$ is a {\it local quadratic transform of $R$}. If $P$ is chosen to be the prime ideal ${\ff m}R[{\ff m}x^{-1}]$, then the ring $R_1$ is a DVR and is the {\it order valuation ring} of $R$, so named because it is the valuation ring of the valuation $v:F \rightarrow {\mathbb{Z}} \cup \{\infty\}$ that restricts  on $R$ to $$v(r) = \max\{i:r \in {\ff m}^i\} {\mbox{ \rm for all }} r \in R.$$  

If $R$ is a regular local ring of Krull dimension two, then the union of a sequence $\{R_i\}_{i=0}^\infty$ of  distinct iterated local quadratic transforms  is  a valuation ring, and  every valuation overring of $R$ that dominates $R$ arises in this way \cite[Lemma 12]{Ab}. It follows that if $R$ has Krull dimension two, then every valuation overring  of $R$ that dominates $R$ is a patch limit point of the order valuation rings of the $R_i$'s. 

By contrast, in dimension three and higher, a valuation overring of the regular local ring $R$ that dominates $R$ need not be a union of a sequence of iterated local quadratic transforms \cite[Examples 4.7 and~4.17]{Sha}. However, 
 it is proved in \cite[Corollary 5.3]{HLOST} that  if $R$ is a regular local ring and  $\{R_i\}_{i=0}^\infty$ is a sequence of distinct iterated local quadratic transforms of $R$ of Krull dimension at least two, then the set of   order valuation rings of the $R_i$ has a unique patch limit point in the Zariski-Riemann space of the quotient field. This valuation ring is termed the {\it boundary valuation ring}  of the sequence $\{R_i\}$.
 
  It is 
shown in \cite[Theorem 5.4]{HLOST} that the union $S=\bigcup_{i=1}^\infty R_i$ is the intersection $S = V \cap T$ of the boundary valuation ring $V$ of $\{R_i\}$ and the {\it Noetherian hull} $T$ of $S$, the smallest Noetherian overring of $R$ containing $S$.    
Moreover, the boundary valuation ring  $V$ has rank at most two \cite[Theorem 6.4 and Corollary 8.6]{HOT}. 
 In the cited reference the valuation for the boundary valuation ring $V$ is expressed asymptotically using limits of invariants associated to $\{R_i\}$.  
Applying Theorem~\ref{finite lim case}, we can also express $V$ as a localization of the intersection of the order valuation rings of the $R_i$. This is a consequence of the next corollary. 



\begin{corollary} \label{RLR} Let $R$ be a regular local ring of Krull dimension at least two.  Let $\{R_i\}_{i=0}^\infty$ be a sequence of overrings of $R$ such that $R = R_0$ and for each $i\geq 0$, $R_{i+1}$ is a local quadratic transform of $R_i$ of Krull dimension at least two.  For each $i$, let $V_i$ be the order valuation ring of $R_i$.  Then the ring $A = \bigcap_{i=1}^\infty V_i$ is a Pr\"ufer domain and  the boundary valuation ring of $\{R_i\}_{i=0}^\infty$ is  the localization of $A$ at the only maximal ideal of $A$ that is not  finitely generated and of height one.
\end{corollary} 

\begin{proof}  Let $X = \{V_i :i\geq 0\}$, and let $V$ be the boundary valuation ring for $\{R_i\}$.   As discussed before the theorem, $\lim X = \{V\}$. By Theorem~\ref{finite lim case}, $A$ is a Pr\"ufer domain since every valuation ring in $X$ has rank one.   Now $\patch(X) = X \cup \lim(X) = X \cup \{V\}$. Since $A$ is a Pr\"ufer domain with quotient field $F$, the set of valuation overrings of $A$ is $\inv(X) = \gen(\patch(X))$ \cite[Proposition 5.6(5)]{OZR}. Thus $\inv(X) = \gen(X) \cup \gen(\{V\}) = X \cup \gen(\{V\})$, where the last equality follows from the fact that $X$ consists of rank one valuation rings. Therefore the valuation rings between $A$ and $F$ that are minimal over $A$ must be among $X \cup \{V\}$.  Let $M$ be a maximal ideal of $A$. Since $A$ is a Pr\"ufer domain, $A_M$ is a valuation domain minimal over $A$, and hence $A_M \in X \cup \{V\}$.  

Suppose first that $A_M \ne V$. Then $A_M$ is a patch isolated point in $X$, which, as in the proof of Theorem~\ref{finite lim case}, implies $M$ is the radical of a finitely generated ideal. Since also $A_M$, as a DVR, has a principal maximal ideal, it follows easily that $M$ is finitely generated. 

On the other hand, suppose $A_M = V$.  If $M$ is the radical of a finitely generated ideal, we have as in the proof of Theorem~\ref{finite lim case} that $A_M$ is isolated in $\patch(X)$ in the patch topology, contrary to the fact that $V \in \lim(X)$. Thus $M$ is not the radical of a finitely generated ideal. 

Finally, to see that $M:= {\ff M}_V \cap A$ is a maximal ideal of $A$, let $N$ be a maximal ideal of $A$ containing $M$.  Suppose  by way of contradiction that $M \subsetneq N$. Then $V \ne A_N$, so that 
since $A_N \in X \cup \{V\}$, we have  $A_N \in X$ and hence $N$ has height one. But
 $A_N \subseteq A_M \subseteq V \subsetneq F$. (Since each $V_i$ dominates $R$ and $V \in \lim(X)$, Proposition~\ref{union} implies $V$ dominates $R$ and hence $V \subsetneq F$.) Since $A_N$ is a rank one valuation ring, this forces $A_N = V$, a contradiction. Therefore, 
$M$ is  a maximal ideal of $A$
%
and the proof is complete. 
\end{proof}




\begin{thebibliography}{99}












\bibitem{Ab} S. Abhyankar, On the valuations centered in a local domain,
 Amer. J. Math. {78} (1956), 321-348.


















\bibitem{DF} D. Dobbs and M. Fontana, Kronecker function rings and abstract Riemann surfaces.  J. Algebra  {99}  (1986),  no. 1, 263--274.


\bibitem{Eng} R. Engelking, {\it General Topology}. Revised and completed edition. Heldermann Verlag, Berlin, 1989.


\bibitem{EP} A. J. Engler and A. Prestel, {\it Valued fields},
Springer-Verlag, 2005.





\bibitem{Fab} A. Fabbri, Integral domains having a unique Kronecker function ring, J. Pure Appl. Algebra 215 (2011), 
 1069--1084.




\bibitem{Fin}  C. Finocchiaro, Spectral spaces and ultrafilters, Comm. Algebra 42 (2014), no. 4, 1496--1508.

\bibitem{FFL} C. Finnochario, M. Fontana and K. A. Loper, The constructible topology on spaces of valuation domains, Trans. Amer. Math. Soc. 365 (2013), 6199--6216.









\bibitem{FHP} M. Fontana, J. Huckaba, and I. Papick, {\it
Pr\"{u}fer Domains}, Marcel Dekker, New York, 1997.




\bibitem{FS} L. Fuchs and L. Salce, {\it Modules over
non-Noetherian domains}, Mathematical Surveys and Monographs, v. 48.
American Mathematical Society. 2001.




\bibitem{G} R. Gilmer, {\it Multiplicative ideal theory,}
Queen's Papers in Pure and Applied Mathematics, No. 12, Queen's
University, Kingston, Ont. 1968.

\bibitem{G2} R. Gilmer, Two constructions of Pr\"ufer domains, J. Reine Angew.
Math. {239/240} (1969), 153--162.








\bibitem{Griffin} M. Griffin, Rings of Krull type, J. reine angew. Math. 229 (1968), 1--27.

\bibitem{Griffin2} 
M. Griffin, Families of finite character and essential valuations, Trans. Amer. Math. Soc.130
(1968), 75--85.



\bibitem{Hart} R.~Hartshorne, {\it Algebraic geometry}. Graduate Texts in Mathematics, No. 52. Springer-Verlag, New York-Heidelberg, 1977. 







\bibitem{HLOST} W. Heinzer, K.~A. Loper, B. Olberding, H.  Schoutens and M. Toeniskoetter, 
Ideal theory of infinite directed unions of local quadratic transforms,   J. Algebra
 474 (2017), 213--239.
 
\bibitem{HOT} W. Heinzer, K.~A. Loper, B. Olberding, H.  Schoutens and M. Toeniskoetter, 
Asymptotic properties of infinite directed unions of local quadratic transforms, 	J. Algebra
 479 (2017),  216--243.
 






\bibitem{HOFC} W. Heinzer and J. Ohm,
Defining families for integral domains of real finite character,
Canad. J. Math.  {24}  (1972), 1170--1177.










\bibitem{Hoc} M.~Hochster, Prime ideal structure in commutative rings, Trans. Amer. Math. Soc. 142 (1969), 43--60.

















\bibitem{LM} M. Larsen and P. J. McCarthy, {\it Multiplicative theory of ideals,}
 Pure and Applied Mathematics, Vol. 43. Academic Press, New York-London, 1971.



   
    
    
    

\bibitem{LT} K. A. Loper and F. Tartarone, A classification of the
integrally closed rings of polynomials containing ${\mathbb{Z}}[x]$,
J. Commutative Algebra {1} (2009), 91--157.






\bibitem{Mat} H. Matsumura, {\it Commutative ring theory,} Cambridge
Studies in Advanced Mathematics 8, Cambridge University Press, 1986.


\bibitem{Kap} I.~Kaplansky,  {\it Commutative Rings}, Allyn and Bacon, Boston, 1970.

\bibitem{Krull} W. Krull, \"Uber die Zerlegungder Hauptideal ein allgemeinen Ringen, Math. Ann. 105 (1931),
no. 1, 1--14.



\bibitem{N} M. Nagata, {\it Local rings}, Interscience Tracts in
Pure and Applied Mathematics, No. 13, John Wiley \& Sons, New
York-London 1962.

\bibitem{Ohm} 
J. Ohm, Some counterexamples related to integral closure in D[[x]], Trans. Amer. Math. Soc.
122 (1966) 321--333.










































\bibitem{OGeo} B. Olberding, On the geometry of Pr\"ufer intersections of valuation rings, Pacific J. Math. 273 (2015), No. 2, 353--368.


\bibitem{OGraz} B. Olberding,  
Topological aspects of irredundant intersections of ideals and valuation rings, in {\it Multiplicative Ideal Theory: Commutative and non-Commutative Perspectives}, Springer, 2016, 277--307.  





\bibitem{OZR} B. Olberding, Affine schemes and topological closures in the Zariski-Riemann space of valuation rings, J. Pure Appl. Algebra {219} (2015), no. 5, 1720--1741.

\bibitem{OCpt} B. Olberding, 
A principal ideal theorem for compact sets of rank one valuation rings, J. Algebra, to appear.


\bibitem{OR} B. Olberding and M. Roitman, The minimal number of generators of an invertible ideal.  {\it Multiplicative ideal theory in commutative algebra},  349--367, Springer, New York, 2006. 

\bibitem{Pirtle} 
E. Pirtle, Families of valuations and semigroups of fractionary ideal classes, Trans. Amer. Math. Soc. 144 (1969), 427--439.


\bibitem{Roq} P. Roquette, Principal ideal theorems for holomorphy rings
 in fields,  J. Reine Angew. Math. {262/263} (1973), 361--374.
 
 \bibitem{Rush} D. Rush, B\'ezout domains with stable range 1, 
J. Pure Appl. Algebra 158 (2001), no. 2--3, 309--324. 





\bibitem{ST} 
N. Schwartz and M. Tressl, Elementary properties of minimal and maximal points in Zariski
spectra, J. Algebra, 323 (2010), 698--728.










\bibitem{Sha} 
D. Shannon, Monoidal transforms of regular local rings. Amer. J. Math. 95 (1973), 294--320.


\bibitem{W} S. Willard, {\it General topology}, Addison-Wesley, 1970.







\bibitem{ZS} O. Zariski and P. Samuel, {\it Commutative algebra}. Vol. II.
 Graduate Texts in Mathematics, Vol. 29. Springer-Verlag, New
York-Heidelberg, 1975.



\end{thebibliography}
\end{document}